\documentclass[a4paper,12pt]{amsart}

\usepackage{tikz-cd}

\usepackage{amssymb,graphicx}
\usepackage{color}
\usepackage{upgreek}
\usepackage{tensor}
\DeclareSymbolFont{script}{U}{eus}{m}{n}
\DeclareMathSymbol{\Wedge}{0}{script}{"5E}
\DeclareMathAlphabet{\mathrmsl}{OT1}{cmr}{m}{sl}
\newcommand{\R}{\mathbb R}
\newcommand{\wt}[1]{\widetilde{#1}}

\newcommand{\weg}[1]{}

\newcommand{\D}{\partial}

\newcommand{\rank}{\operatorname{rank}}

\newcommand{\const}{\mathrm{const}}

\theoremstyle{plain}
\newtheorem{thm}{Theorem}[section]
\newtheorem*{thm*}{Theorem}
\newtheorem{lemma}[thm]{Lemma}
\newtheorem{proposition}[thm]{Proposition}
\newtheorem{cor}[thm]{Corollary}

\theoremstyle{definition}
\newtheorem{defn}{Definition}[section]

\theoremstyle{remark}
\newtheorem{rem}{Remark}[section]

\usepackage{soul}
\setulcolor{blue}

\title[Symplectic invariants for parabolic orbits and cusp singularities]{Symplectic invariants for parabolic orbits and cusp singularities of integrable systems with two degrees of freedom}
\author{Alexey Bolsinov, Lorenzo Guglielmi and Elena Kudryavtseva}

\address{Department of Mathematical Sciences,
 Loughborough University,
 LE11 3TU, UK  and  Faculty of Mechanics and Mathematics, Moscow State University, Moscow 119991, Russia}
 \email{A.Bolsinov@lboro.ac.uk} 
\address{SISSA, Via Bonomea, 265 -- 34136 Trieste, ITALY}
\email{lgugliel@sissa.it}
\address{Faculty of Mechanics and Mathematics, Moscow State University, Moscow 119991, Russia}
\email{eakudr@mech.math.msu.su}

\begin{document}
\maketitle

\begin{abstract}
We discuss normal forms and symplectic invariants of parabolic orbits and cuspidal tori in integrable Hamiltonian systems with two degrees of freedom.  Such singularities appear in many integrable systems in geometry and mathematical physics and can be considered as the simplest example of degenerate singularities. We also suggest some new techniques which apparently  can be used for studying symplectic invariants of degenerate singularities of more general type. 
\end{abstract}


\section{Introduction} \label {sec:1}

An integrable Hamiltonian system on a symplectic manifold $(M^{2n}, \Omega)$ is defined by $n$ pairwise commuting functions $F_1, \dots, F_n$ which are independent on $M^{2n}$ almost everywhere.  We will consider the case  $n=2$ and denote such a pair of commuting functions by $H$ and $F$  ($H$ is usually considered as the Hamiltonian and $F$ as an additional first integral).   Under the above assumptions,  on $M^4$  we can introduce the structure of a singular Lagrangian fibration whose fibers are, by definition, common level surfaces  $\mathcal L_{h,f} = \{ H = h, \ F=f\}$,   $(h,f)\in \R^2$  (or their connected components).  We will assume that all the fibers are compact  (unless we study local properties of a system). The functions  $H$ and $F$ also define a Hamiltonian  $\R^2$-action on $M^4$.

According to Liouville theorem,   regular compact connected fibers are 2-dimensional Lagrangian tori of dimension 2  which coincide with orbits of the $\R^2$-action.    We say that a fiber $\mathcal L_{h,f}$ is singular if it contains a singular point, i.e., a point $P$ such that $dH(P)$ and $dF(P)$  are linearly dependent.  Equivalently, we may say that  $\mathcal L_{h,f}$ is singular if it contains an orbit of a non-maximal dimension, i.e., $1$ or $0$.  A general problem of the theory of singularities of integrable systems is to describe the topology of singular fibers and their saturated neighborhoods (similarly for singular orbits). Notice that the fact that $F$ and $H$ commute makes this theory rather different as compared to the classical  singularity theory for smooth maps.  

Saying ``describe''  we may mean at least three different settings:  topological, smooth and symplectic.  For instance, saying that two given singularities   (points, orbits or fibers)  are symplectically equivalent  we mean the existence of a fiberwise symplectomorphism between their neighborhoods.  Throughout the paper, in addition we will assume that all the objects we are working with are real (or complex) analytic.

In this paper we discuss just one particular type of singularitites,  namely parabolic orbits and cuspidal tori (speaking informally,  a cuspidal torus is  a compact singular fiber that contains one parabolic orbit and no other singular points). 

Recall that typical  (non-degenerate)  singular orbits in integrable Hamiltonian systems can be of two different types: elliptic and hyperbolic.  
In integrable systems of two degrees of freedom, we may  very often  observe a  transition from {\it elliptic} to {\it hyperbolic} in a smooth one-parameter family of singular orbits.  At the very moment of transition,  the orbit becomes degenerate and of parabolic type.  This scenario is rather natural and  {\it parabolic} can be viewed  as the simplest possible type of degenerate singularities.
 
 Another important property of parabolic orbits is their stability  under small integrable perturbations (this follows from \cite[Sec. 6.1, 6.6, 7.3]{AVG}).  This is one of the reasons why such orbits can be observed in many examples  of integrable Hamiltonian systems:    Kovalevskaya top \cite{BolsFomRichter}, other integrable cases in rigid body dynamics including Steklov case, Clebsch case, Goryachev--Chaplygin--Sretenskii case, Zhukovskii case, Rubanovskii case and Manakov top on so(4) \cite{BolsFomBook}, as well as systems invariant w.r.t. rotations \cite{Kantonistova}, \cite{KozlovOshemkov},  see also examples discussed in \cite{Efstathiou}, \cite{DullinIvanov2}. Unlike non-degenerate singularities, however, in the literature on topology and singularities of integrable systems there are only few papers devoted to degenerate singularites including parabolic ones.  We refer, first of all, to  the following six ---  L. Lerman, Ya. Umanskii \cite{LermanUmanskii2},  V. Kalashnikov \cite{Kalashnikov}, N. T. Zung \cite{ZungDegen}, H. Dullin, A. Ivanov \cite{DullinIvanov2},  K. Efstathiou, A. Giacobbe \cite{Efstathiou} and Y. Colin de Verdi\`ere \cite{CdV} --- which we consider to be very important in the context of general classification programme for bifurcations occurring in integrable systems.  

It is well known that from the smooth point of view,
all parabolic orbits are equivalent,  i.e. any two parabolic orbits admit fiberwise diffeomorphic neighborhoods   (Lerman-Umanskii \cite{LermanUmanskii1, LermanUmanskii2}, Kalashnikov \cite{Kalashnikov}).  The same is true for cuspidal tori  \cite{Efstathiou}.  The simplest model for a parabolic singularity is as follows.

Consider the direct product of $\R^3$ with coordinates $x,y,\lambda$ and a circle  $S^1$  parametrized by $\varphi \mod 2\pi$  and two functions on this product $\R^3 \times S^1$:
\begin{equation}
H = x^2 + y^3 + \lambda y \quad\mbox{and} \quad F = \lambda.
\end{equation}
They commute with respect to the symplectic form
\begin{equation}
\Omega = dx\wedge dy + d\lambda \wedge d\varphi.
\end{equation}

The curve $\gamma_0(t)=(0, 0, 0, t)$ is a parabolic orbit of an integrable Hamiltonian system defined by commuting functions $H$ and $F$. However, in general,  we cannot assume that these coordinates $x,y,\lambda, \varphi$ are canonical  (in other words, the formula for $\Omega$ could be different).

The starting point of the present paper was the following question. We know that elliptic and hyperbolic orbits have no symplectic invariants \cite{ZungMiranda}. In other words,  for any elliptic or hyperbolic  (with orientable or non-orientable separatrix diagram) orbit there exists a symplectic canonical form, one and the same for all orbits of a given type (see, e.g., \cite{BBM_UMN}). Is the same true for parabolic orbits or they admit non-trivial symplectic invariants?

It appears that non-trivial symplectic invariants do exist  (a very simple invariant is given by Proposition \ref{prop:3.3.11}).  Moreover, we show that all symplectic invariants of parabolic orbits can be expressed in terms of action variables (Theorem \ref{thm:5.6}).   The next natural step would be to extend a fiberwise symplectomorphism between tubular neighborhoods of two parabolic orbits to saturated neighborhoods of the cuspidal tori that contain these orbits. This is done in Section \ref{sec:6}:   Theorems \ref{th:6.1} and \ref{th:6.2} (see also Remark \ref{rem:converse}) give necessary and sufficient conditions for symplectic equivalence of cuspidal tori.  The latter theorems basically say the only symplectic semi-local invariant of a cuspidal torus is the canonical integer affine structure on the base of the corresponding singular Lagrangian fibration.  In other words, cuspidal tori satisfy the following principle formulated  in   \cite{BITOpenProb}:

{\it Let $\phi : M \to B$ and $\phi': M' \to B'$ be two singular Lagrangian fibrations. If $B$ and $B'$ are affinely equivalent (as stratified
manifolds with singular integer affine structures), then these Lagrangian fibrations are fiberwise symplectomorphic.}

Also we would like to notice that although parabolic singularities are rather simple and specific,  some techniques developed and used in this paper are quite general and can be used for analysis of more complicated singularities.  They also can be generalised to the case of many degrees of freedom.

Acknowledgements:    This work was supported by the Russian Science Foundation (project no. 17-11-01303).   We are very grateful to D.\,Guzzetti,  A.\,Varchenko, S.\,Nemirovskii, N.\,T.\,Zung, A.\,Izosimov  for valuable comments and discussions.


\section{Definition of parabolic singularities.  Canonical form with no symplectic structure} \label{sec:2}

We begin with the definition of parabolic orbits following \cite{Efstathiou} \footnote{One essential condition, in our opinion, is missing in  \cite{Efstathiou}.  In  Definition \ref{def:parabolic} below we make a necessary modification by adding Condition (iii).}. Let $H$ and $F$ be a pair of Poisson commuting real-analytic functions on a  real-analytic symplectic manifold $(M^4, \Omega) $.  They define a Hamiltonian $\R^2$-action  (perhaps local) on $M^4$.  The dimension of the $\R^2$-orbit through a point $P\in M^4$ coincides with the rank of  the differential of the momentum map $\mathcal F=(H,F): M^4 \to \R^2$ at this point and we are interested in one-dimensional orbits.  Without loss of generality throughout the paper we will assume that $dF(P)\ne 0$.   Consider the restriction of $H$ onto the three-dimensional level set of $F$ through $P$, that is,  $H_0:=H|_{\{F=F(P)\}}$.  We assume that the rank of $d\mathcal F$ at the point $P$ equals one.  This is equivalent to any of the following:
\begin{itemize}

\item $P$ is a critical point of $H_0$;

\item there exists a unique $k\in \R$ such that $dH(P) = k dF(P)$, in particular, $P$ is a critical point of $F-kH$.
\end{itemize}

These properties hold true for each singular point $P$ of rank one of the momentum mapping $\mathcal F=(H,F)$ under the condition that $dF(P)\ne 0$.

\begin{defn} \label{def:parabolic}
A point $P$ (and the corresponding $\R^2$-orbit through this point) is called {\it parabolic} if the following conditions hold:
\begin{enumerate}
\item[(i)] the quadratic differential $d^2H_0(P)$ has rank 1; 
\item[(ii)] there exists a vector $v \in \ker d^2H_0(P)$ such that $v^3H_0\ne 0$ (by $v^3H_0$ we mean the third derivative of $H_0$ along the tangent vector $v$ at $P$);
\item[(iii)] the quadratic differential $d^2(H-kF)(P)$ has rank 3, where $k$ is the real number determined by the condition $dH(P)=kdF(P)$.
\end{enumerate}
\end{defn}

\begin{rem}{\rm  
In this definition, we use the third derivative of a function along a tangent vector which, in general,  is not well defined.  In our special case, however,  this derivative makes sense as $dH_0(P)=0$ and $v\in \ker d^2H_0(P)$.  These two properties allows us to define it as follows:
$$
v^3 (H_0) = \frac{d^3}{dt^3} |_{t=0} H_0 (\gamma (t)), 
$$
where $\gamma (t)$ is an arbitrary  curve on the hypersurface $\{ F = F(P)\} $ such that  $\gamma(0)=P$, $\frac{d\gamma}{dt} (0) = v$. The result does not depend on the choice of $\gamma(t)$. Indeed,  
$$
\frac{d^3}{dt^3} H_0(\gamma(t)) =   d^3H_0 (\gamma', \gamma',\gamma') + 3d^2H_0(\gamma',\gamma '') + dH_0(\gamma''') = d^3H_0 (\gamma', \gamma',\gamma') = d^3H_0 (v,v,v),
$$ 
as $dH_0=0$ and $\gamma' \in \ker d^2H_0$. This computation also shows that the third differential $d^3H_0$ is a well-defined cubic form on $\ker d^2H_0$ so that condition (ii) is equivalent to the fact that the third differential $d^3H_0$ does not vanish on $\ker d^2 H_0$  (at the point $P$). 
}\end{rem}

\begin{rem}\label{rem:new1}
It can be checked that in  Definition \ref{def:parabolic}, we may replace $H$ and $F$ by any other independent functions $\wt H = \wt H (H, F)$,  $\wt F = \wt F(H,F)$ such that $d\wt F(P) \ne 0$. In other words, the property of {\it being parabolic}  refers to a singularity of  the momentum mapping $\mathcal F:  M^4 \to \R^2$ and does not depend on the choice of local coordinates in a neighborhood of $\mathcal F(P) \in \R^2$.  Necessary details can be found in Appendix, see Proposition \ref{prop:ap1}.
\end{rem}

The following statement describes the structure of the singular Lagrangian fibration in a neighborhood of a parabolic point $P$.  As we are mostly interested in this fibration (rather than specific commuting functions $H$ and $F$), we allow ourselves to replace $H$ with $\widetilde H = \widetilde H(H,F)$ where $\frac{\partial \widetilde H}{\partial H}\ne 0$ and to shift  and change the sign of $F$, so that $\widetilde H$ and $\widetilde F= \pm F+\mathrm{const}$  still commute and define the same Lagrangian fibration as $H$ and $F$.  Notice that according to Remark \ref{rem:new1},  $P$ is parabolic for $\widetilde H$ and $\widetilde F$.

\begin{proposition} \label{prop:2.1}
Locally in a neighborhood of $P$  there exist a transformation 
\begin{equation} \label{eq:transf1}
\begin{aligned}
\widetilde H &= \widetilde H(H,F),  \quad\mbox{with } \frac{\partial \widetilde H}{\partial H}\ne 0, \\
\widetilde F &= \pm F + \mathrm{const},
\end{aligned}
\end{equation}
and a local coordinate system $x,y,\lambda, \varphi$ such that  $(x,y,\lambda,\varphi)|_P=(0,0,0,0)$ and
\begin{equation} \label{eq:canon1}
\widetilde H = \widetilde H(H,F)= x^2 + y^3 + \lambda y \quad\mbox{and} \quad \widetilde F = \lambda.
\end{equation}
\end{proposition}

\begin{rem}
We do not require that this coordinate system is canonical and, in this view,  Proposition \ref{prop:2.1} describes a normal form of a parabolic singularity in the sense of Singularity Theory with no symplectic structure involved.  This statement is local and we do not need to assume that the orbit through $P$ is closed. Later on, the variable $\varphi$ will be one of the angle variables defined modulo $2\pi$,  but here $\varphi$ just  belongs to a certain interval.
\end{rem}

\begin{proof}   
The proof of this statement if well known but we still want to briefly explain some of its steps to reveal  important underlying phenomena. The first step is to find $x,y,\lambda,\varphi$ without touching $H$ and $F$. 

\begin{lemma} \label {lem:HF}
Under the above assumptions,  there exist local coordinates $x,y,\lambda,\varphi$ such that $(x,y,\lambda,\varphi)|_P=(0,0,F(P),0)$ and
\begin{equation}
\label{eq:canon2}
H = \pm (x^2  + y^3 + b(\lambda) y + a(\lambda)),    \quad  F=\lambda,
\end{equation}
where $a(\lambda)$ and $b(\lambda)$ are real-analytic functions with $b(F(P))=0$, $b'(F(P))\ne0$.
\end{lemma}

\begin{proof} 
Without loss of generality, we assume that $H(P)=F(P)=0$.
First of all we need to kill one dimension using the fact that $H$ and $F$ Poisson commute.  Since  $dF(P)\ne 0$ we can choose a canonical coordinate system $p_1, q_1, p_2, q_2$ such that $F=q_2$.  Since $H$ and $F$ commute,  we conclude that $H$ does not depend on $p_2$, i.e., $H=H(p_1, q_1, q_2)$. Thus, $p_2$ does not play any role, so we may forget about it and continue working with $p_1, q_1, q_2$.

Let us now think of $H$ as a function of two variables $q_1$ and $p_1$ depending on $q_2 = \lambda$ as a parameter.
We have $\partial H/\partial p_1|_P=\partial H/\partial q_1|_P=0$ and, without loss of generality, $\partial^2H/\partial p_1^2|_P\ne0$.
We are now in a quite standard situation in singularity theory.

By a parametric version of the Morse lemma, the function $H$ can be written as $H=\pm(x^2+f(q_1,\lambda))$, for some new local variable $x=x(p_1,q_1,\lambda)$ such that $x|_P=0$ and $\D x/\D p_1\ne0$. 
Now, condition (ii) of the definition of a parabolic point is satisfied if and only if the function $f(q_1,0)$ in one variable $q_1$ has order $3$ at the point $q_1|_P$. Hence, this function can be written as $\hat y^3$ for some variable $\hat y=\hat y(q_1)$ with $\hat y(q_1(P))=0$.

Now the function $f(q_1,\lambda)$ is a 1-parameter ``deformation'' of the function $f(q_1,0)=\hat y^3$ with the parameter $\lambda$. It follows from \cite [Sec. 8.2, Theorem or Example]{AVG} that the deformation $\hat y^3+\lambda_2\hat y+\lambda_1$ is right-infinitesimally versal. By the versality theorem \cite[Sec. 8.3]{AVG}, it is right-versal (for a definition of a versal deformation, see \cite[Sec. 8.1]{AVG}). Since any deformation is right-equivalent to a deformation induced from the right-versal one, we have $f(q_1,\lambda)=y^3+b(\lambda)y+a(\lambda)$ for some real-analytic functions $y=y(\hat y,\lambda)$, $a(\lambda)$ and $b(\lambda)$ such that $y(\hat y,0)=\hat y$, $a(0)=b(0)=0$. Since, by assumption, the quadratic differential $d^2(H-kF)(P)$ has rank 3, we have $b'(0)\ne0$.
So, we obtain the representation \eqref{eq:canon2}.
\end{proof}

Later on we will need to rearrange leaves of our singular Lagrangian fibration by using some transformations of the form  (the fibration itself remains unchanged) 
\begin{equation}
\label{eq:transbase}
H \mapsto \widetilde H=\widetilde H(H,F),  \  F \mapsto \widetilde F=\widetilde F(H,F).
\end{equation}
So we need to understand if such a transformation (acting on the base of the Lagrangian fibration) can be realised by a fiberwise analytic diffeomorphism upstairs.  In other words, we want to know which of transformations \eqref{eq:transbase} are liftable.

Let us look at the (local) bifurcation diagram (i.e.\ the set of critical values) of the map defined by $H$ and $F$ from \eqref{eq:canon2}. This bifurcation diagram is as follows, for a $+$ sign in \eqref{eq:canon2}:
$$
\Sigma = \left\{ \bigl(H-a(F)\bigr)^2 = -\frac{4}{27} b(F)^3 \right\} \subset \R^2 (H,F), 
$$
and it has a cusp at the point $(H(P),F(P))$ that splits $\Sigma$ into two smooth branches, $\Sigma_{\mathrm{ell}}$ and $\Sigma_{\mathrm{hyp}}$, corresponding to one-parameter families of elliptic and hyperbolic orbits. The bifurcation diagram for $a(\lambda)=0$ and $b(\lambda)=\lambda$ is shown on Fig. 3. 
Notice that our choice of the + sign in \eqref{eq:canon2} simply means that the monotone function $H-kF|_\Sigma$ increases w.r.t. the orientation of $\Sigma$ from $\Sigma_{\mathrm{ell}}$ to $\Sigma_{\mathrm{hyp}}$.

It can be easily seen (see the proof of Proposition \ref{prop:lift} below) that this bifurcation diagram allows us to reconstruct  both functions $a(\lambda)$ and $b(\lambda)$.    We will use this observation to prove the following

\begin{proposition}\label{prop:lift}
Assume we have two parabolic singularities defined by functions $H, F$ at a point $P$ and $\widetilde H, \widetilde F$ at a point $\widetilde P$ respectively.  A map (local  analytic diffeomorphism)
$$
\phi :  \R^2 (H, F) \to \R^2 (\widetilde H, \widetilde F)
$$
is liftable if and only if $\phi$ transforms the bifurcation diagram  of $(H,F)$ to that of $(\widetilde H, \widetilde F)$, i.e. $\phi (\Sigma) = \widetilde \Sigma$, together with its partition into elliptic and hyperbolic branches. In other words, the condition $\phi (\Sigma) = \widetilde \Sigma$ is necessary and sufficient for the existence of a local  analytic diffeomorphism $\Phi$ such that the diagram 
$$
\begin{tikzcd}
M^4 \arrow{r}[]{\Phi}  \arrow{d}{(H,F)}&  {\widetilde{M}}^4  \arrow{d}[]{(\widetilde H,\widetilde F)}\\
\R^2 \arrow{r}{\phi}& \R^2
\end{tikzcd}
$$
is commutative.
\end{proposition}

\begin{proof}
The ``only if'' part is obvious. 

Let us prove the ``if'' part.
Denote $\phi\circ(H,F)$ by $(H_1,F_1)$. Clearly, $\phi$ transforms the bifurcation diagram of $(H,F)$ to that of $(H_1,F_1)$, together with their partitions into elliptic and hyperbolic branches. Hence, the bifurcation diagram $\Sigma_1$ of $(H_1,F_1)$ coincides with the bifurcation diagram $\widetilde\Sigma$ of $(\widetilde H,\widetilde F)$, together with its partition into elliptic and hyperbolic branches.

As shown above,  under the condition that $d\widetilde F(\wt P)\ne0$,  the bifurcation diagram $\wt \Sigma$ of the mapping $\wt{\mathcal F}=(\wt H, \wt F): {\widetilde{M}}^4 \to \R^2(h,f)$ is defined by
$$
\wt \Sigma =\{ (h,f)\in \R^2 ~|~   (h- \tilde a(f))^2 = -\frac{4}{27} \tilde b(f)^3\}
$$
for some functions $a(\cdot)$ and $b(\cdot)$ determined by the canonical form \eqref{eq:canon2}.
Hence $\widetilde\Sigma$ lies entirely in a half-plane $\{(h,f) \mid \widetilde b(f)\le 0\}\subset\R^2(h,f)$  bounded by a line $\{f=\const\}$ through the cusp point $(\widetilde H(\widetilde P),\widetilde F(\widetilde P))$.  Since $\Sigma_1=\widetilde\Sigma$, we conclude that $dF_1(P)\ne0$ as well. 
 
By Lemma \ref {lem:HF}, there exist local (real-analytic) coordinates $x_1,y_1,\lambda_1,\varphi_1$ in a neighborhood $U_1$ of $P$ and coordinates $\widetilde x,\widetilde y,\widetilde \lambda,\widetilde \varphi$ in a neighborhood $\widetilde U$ of $\widetilde P$ such that 
\begin{equation} \label {eq:H1F1}
\begin{aligned}
\eta_1 H_1 &= x_1^2  + y_1^3 + b_1(\lambda_1) y_1 + a_1(\lambda_1),    \quad  & F_1 &= \lambda_1, \\
\widetilde\eta \widetilde H &= \widetilde x^2  + \widetilde y^3 + \widetilde b(\widetilde \lambda) \widetilde y + \widetilde a(\widetilde \lambda),    \quad  &\widetilde F &= \widetilde \lambda,
\end{aligned}
\end{equation}
for some signs $\eta_1,\wt\eta\in\{1,-1\}$. 
The elliptic and hyperbolic branches of $\widetilde\Sigma$ have the form
$$
\begin{aligned}
\widetilde\Sigma_{\mathrm{ell}} &= \left\{\left. (h,f)= \left(\widetilde\eta\left(\widetilde a(f) - 2(-\widetilde b(f)/3)^{3/2}\right),f\right) \, \right|\, \widetilde b(f)<0\right\}, \\
\widetilde\Sigma_{\mathrm{hyp}} &= \left\{\left. (h,f)= \left(\widetilde\eta\left(\widetilde a(f) + 2(-\widetilde b(f)/3)^{3/2}\right),f\right) \, \right|\, \widetilde b(f)<0\right\},
\end{aligned}
$$
in particular, 
$\widetilde\eta h-\widetilde a(f)>0$ on the hyperbolic branch of $\widetilde\Sigma$ and $<0$ on its elliptic branch.
Since similar properties and formulae hold for the elliptic and hyperbolic branches $\Sigma_{1,\mathrm{ell}}$ and $\Sigma_{1,\mathrm{hyp}}$ of $\Sigma_1$, 
moreover $\Sigma_{1,{\mathrm{ell}}}=\widetilde\Sigma_{\mathrm{ell}}$ and $\Sigma_{1,{\mathrm{hyp}}}=\widetilde\Sigma_{\mathrm{hyp}}$,  we obtain the equalities
\begin{equation} \label {eq:eq}
\eta_1=\widetilde\eta, \quad
a_1(\lambda)=\widetilde a(\lambda), \quad
b_1(\lambda)=\widetilde b(\lambda)
\end{equation}
where the equalities of functions hold in a half-neighbourhood $\{\lambda\mid\widetilde b(\lambda)\le0\}$ of the point $\widetilde F(\widetilde P)\in\R$. Since all functions are real-analytic at this point, the equalities of functions in \eqref{eq:eq} hold in an entire neighbourhood.

Define a (real-analytic) diffeomorphism germ $\Phi:(U_1,P)\to(\widetilde U,\widetilde P)$ given by the identity map in the local coordinates $(x_1,y_1,\lambda_1,\varphi_1)$ and $(\widetilde x,\widetilde y,\widetilde \lambda,\widetilde \varphi)$. By \eqref{eq:H1F1} and \eqref{eq:eq}, $\Phi$ transforms $(\widetilde H,\widetilde F)$ to $(H_1,F_1)$, so it 
has the desired property $\phi\circ(H,F)=(H_1,F_1)=(\widetilde H,\widetilde F)\circ\Phi$.
\end{proof}

Proposition \ref{prop:lift}  implies the following

\begin{cor}
\label{cor:2.1}
Let $P$ be a parabolic point for an integrable Hamiltonian system with the momentum mapping $\mathcal F= (H,F): M^4 \to \R^2$. Assume that the local bifurcation diagram $\Sigma \subset \R^2(H,F)$ of $\mathcal F$ takes the standard form 
\begin{equation}
\label{eq:canonbd1}
\Sigma = \left\{ {H}^2 = -\frac{4}{27} {F}^3 \right\}  \quad \mbox{with} \quad
\Sigma_{\mathrm{ell}} = \Sigma \cap \{{H}<0\},\ 
\Sigma_{\mathrm{hyp}} = \Sigma \cap \{{H}>0\}. 
\end{equation}
Then in a neighborhood of a parabolic point  there exists a local coordinate system $(x,y,\lambda,\varphi)$  in which 
$H=x^2 + y^3 + \lambda y$ and $F=\lambda$.
\end{cor}

\begin{proof}   
It is sufficient to notice that the pair of functions $\wt H = x^2 + y^3 + \lambda y$, $\wt F = \lambda$ define a parabolic singular point with the standard bifurcation diagram \eqref{eq:canonbd1}.   According to Proposition \ref{prop:lift}  any other parabolic singularity with the same bifurcation diagram is fiberwise diffeomorphic to this simplest model, moreover,  the map $\phi: \R^2(H,F) \to \R^2(\wt H, \wt F)$ between the bases is defined by $\wt H = H$, $\wt F=F$.  \end{proof}

We are now able to complete the proof of Proposition \ref{prop:2.1}. In view of Corollary \ref{cor:2.1}, it is sufficient to show that by a suitable transformation \eqref{eq:transf1} the bifurcation diagram, together with its partition into elliptic and hyperbolic branches, can be reduced to the standard form \eqref{eq:canonbd1}. 

As shown above, for the original functions $H$ and $F$ the bifurcation diagram is defined by the equation 
$$
\Sigma = \left\{ \bigl(H-a(F)\bigr)^2 = -\frac{4}{27} b(F)^3 \right\}
$$
(here we assume that $H$ in \eqref{eq:canon2} comes with $+$).

Let  $F(P) = f_0$ so that $b(f_0)=0$ and $b'(f_0)\ne0$,  then we can represent $b(\lambda)$ as $b(\lambda) = (\lambda-f_0) c(\lambda)$ with $c(f_0)\ne 0$  and rewrite the equation for $\Sigma$ in the form
$$
\left( \frac{H -a(F)}{|c(F)|^{3/2}}\right)^2 = -\eta_F \frac{4}{27} (F-f_0)^3 
$$
with $\eta_F=c(f_0)/|c(f_0)|$ or, equivalently, 
$$
\Sigma = \left\{ {\widetilde H}^2 = -\frac{4}{27} \widetilde{F}^3 \right\}  \quad \mbox{with} \quad
\Sigma_{\mathrm{ell}} = \Sigma \cap \{\widetilde{H}<0\},\ 
\Sigma_{\mathrm{hyp}} = \Sigma \cap \{\widetilde{H}>0\}, 
$$
for $\widetilde{H} = \dfrac{H-a(F)}{|c(F)|^{3/2}}$ and $\widetilde{F} = \eta_F (F - f_0)$, which coincides with \eqref{eq:canonbd1} as required.  \end{proof}


\section{Description of a neighborhood of a parabolic orbit with  symplectic structure} \label{sec:3}

Our next goal is to describe the symplectic structure $\Omega$ near a parabolic orbit.

An important property of a parabolic orbit is the existence (in real-analytic case) of a free Hamiltonian $S^1$-action in its tubular neighborhood  (N.T. Zung  \cite{ZungDegen}, compare Kalashnikov \cite{Kalashnikov}).
In other words, without loss of generality we may assume that one of the commuting functions, say $F$, generates this $S^1$-action, i.e., the Hamiltonian flow of $F$ is $2\pi$-periodic. 
From the viewpoint of singularity theory, this means that in our case the parameter of the versal deformation is essentially unique and is given by the Hamiltonian of the $S^1$-action  (or in slightly different terms, by the action variable related to the cycle in the first homology group of fibers that corresponds to this $S^1$-action).  The latter interpretation, in particular, means that one of two action variables is a real-analytic function defined on the whole neighborhood $U(\mathcal L_0)$ of $\mathcal L_0$ including singular fibers, where $\mathcal L_0$ denotes the singular fiber (cuspidal torus) containing the parabolic orbit $\gamma_0$.
The action variable $F$ is defined up to changing $F\to\pm F+\const$, and we can (and will) choose $F$ in such a way that $F(P)=0$ and the bifurcation diagram $\Sigma$ is located in the domain $\{ F \le 0\}$.

Basically, what we want to do next is to reduce our Hamiltonian system w.r.t. this action.  We shall think of $F$ as a parameter and denote it by $\lambda$  as above.
In particular, now we can choose a 
coordinate system $x,y,\lambda,\varphi$ in a tubular neighborhood $U(\gamma_0)$ of $\gamma_0$ in such a way that the Hamiltonian vector field of $\lambda$ is $\frac{\partial}{\partial \varphi}$. Since $H$ commutes with $F=\lambda$, we conclude that $H=H(x,y,\lambda)$ and we are in the situation discussed in the previous section. If we are only interested in the symplectic topology of the fibration, we are free in the choice of $H$ (in contrast to $F$ which is essentially unique), so according to Proposition \ref{prop:2.1} we may assume  without loss of generality 
that $H= x^2 + y^3 + \lambda y$. However, these coordinates are not canonical, so that (in the tubular neighborhood $U(\gamma_0)$) the symplectic structure takes the following form (here we take into account the condition that $\Omega$ is closed and the Hamiltonian vector field of $\lambda$ is $\frac{\partial}{\partial \varphi}$  or, equivalently,    $i_{\D/\D \varphi} \Omega = -d\lambda$):  
\begin{equation} \label{eq:Omega}
\begin{aligned}
\Omega &= f(x,y,\lambda) dx\wedge dy + d\lambda\wedge d\varphi +  d\lambda \wedge \bigl(P(x,y,\lambda) dx + Q(x,y,\lambda) dy\bigr)=\\
&=\omega_\lambda + d\lambda\wedge d\varphi + \mbox{ (additional terms)}.
\end{aligned}
\end{equation}

The form $\omega_\lambda = f(x,y,\lambda)dx\wedge dy$  can be considered as the restriction of $\Omega$ onto the common level of $\lambda$ and $\varphi$ (we assume that $\varphi=0$ but $\lambda$  varies and is considered as a parameter).   The other interpretation of $\omega_\lambda$ is that it is the one-parameter family of symplectic forms obtained from $\Omega$ by the reduction w.r.t. the Hamiltonian $S^1$-action  (or, using old-style terminology, w.r.t. the cyclic variable $\varphi$). Here is a more formal statement.

\begin{proposition} \label{pro:coord}
In a tubular neighborhood of a parabolic orbit $\gamma_0$ we can choose a coordinate system $x,y,\lambda,\varphi$ {\rm(}with  $\varphi\,\operatorname{mod} 2\pi \in \R / 2\pi\mathbb Z${\rm)}  such that $(x,y,\lambda)|_{\gamma_0} = (0,0,0)$ and our singular Lagrangian fibration is given by two functions
$$
F=\lambda \quad\mbox{and}\quad H=x^2 + y^3 + \lambda y
$$
and the symplectic form 
$$
\Omega = f(x,y,\lambda) dx\wedge dy   +   d\lambda \wedge d\varphi  +  \mbox{ (additional terms) }
$$
as in \eqref{eq:Omega}.  \hfill$\square$
\end{proposition}

\begin{rem}
Without loss of generality we may assume that $f(x,y,\lambda)>0$    in \eqref{eq:Omega}. Indeed,  in order for the latter property to be fulfilled, we only need to replace $x$ with $-x$ if necessary.  We also notice that since $\Omega$ is closed,  formula \eqref{eq:Omega} can be rewritten as
$$
\Omega =  dX(x,y,\lambda) \wedge dy + d\lambda\wedge d\wt\varphi
$$
for a certain real-analytic function $X(x,y,\lambda)$ with $\frac{\D X}{\D x}>0$ and $\wt \varphi = \varphi + R(x,y, \lambda)$ for some real-analytic function $R(x,y,\lambda)$.
\end{rem}

It follows from Proposition \ref{pro:coord} that the function $F$ is uniquely defined (being a generator of the $S^1$-action), but $H$ is not.  However $H$ cannot be chosen arbitrarily because the bifurcation diagram for $F$ and $H$  must be of a very special form, namely \eqref{eq:canonbd1}.  If this condition is fulfilled then $H$ is allowed and, using Corollary \ref{cor:2.1},   we can modify Proposition \ref{pro:coord}  in the following way.  

\begin{proposition}\label{pro:coord2}
Consider a tubular neighborhood of a parabolic trajectory.  Let $H$ and $F$ be two functions defining our fibration and satisfying the following conditions:
\begin{itemize}
\item[(i)] the bifurcation diagram of $(H,F)$ is canonical, i.e., as in  \eqref{eq:canonbd1};
\item[(ii)] $F$ is  $2\pi$-periodic, i.e., is a generator of a free Hamiltonian $S^1$-action. 
\end{itemize}
Then there exists a coordinate system $(x,y,\lambda,\varphi)$  as in Proposition \ref{pro:coord}. \hfill$\square$
\end{proposition}

\begin{rem}\label{rem:about2forms}
It follows from Proposition \ref{pro:coord} that if we are given two integrable systems with parabolic trajectories, we can always find a fiberwise real-analytic diffeomorphism between their tubular neighborhoods that respects the $S^1$-actions and corresponding periodic Hamiltonians.  This means that without loss of generality we may assume that we are given just one single fibration defined by $H$ and $F$ having canonical form \eqref{eq:canon1}  with two different symplectic forms given by \eqref{eq:Omega}  (i.e. such that $H$ and $F$ commute and the Hamiltonian vector field of $\lambda$ is $\frac{\partial}{\partial \varphi}$):
\begin{equation} \label{eq:omega} 
\Omega = \omega_\lambda + d\lambda\wedge d\varphi + \mbox{ (additional terms)}
\end{equation}
and
\begin{equation} \label{eq:tildeomega}
\widetilde\Omega=\tilde \omega_\lambda + d\lambda\wedge d\varphi + \mbox{ (additional terms)}.
\end{equation}

We still have two different integrable systems but after the above ``pre-identification'' they have many common properties. Namely,

\begin{enumerate}
\item[(i)] They have a common local coordinate system $(x,y,\lambda,\varphi)$ from Proposition \ref{pro:coord};
\item[(ii)] $F=\lambda$ is a $2\pi$-periodic integral for the both systems;
\item[(iii)] The $S^1$-actions defined by $F$ for $\Omega$ and $\wt\Omega$  coincide  (i.e.,  $X_F = \wt X_F=\frac{\partial}{\partial  \varphi}$ where $X_F$ and $\wt X_F$ denote the  Hamiltonian vector fields generated by $F$ w.r.t. $\Omega$ and $\wt\Omega$  respectively);
\item[(iv)] The bifurcation diagrams of these two systems coincide;
\item[(v)] The orientations and co-orientations of the parabolic trajectory  $\gamma_0(t){=}(0,0,0,\varphi{=}t)$ induced by $\Omega$ and $\wt\Omega$ coincide  (see Section~\ref{sec:5}, Theorem \ref{thm:5.4}). 
\end{enumerate}

We need to find out whether $\Omega$ can be transformed to $\widetilde\Omega$ by a suitable fiberwise diffeomorphism $\Phi$. First, we  impose a stronger condition on $\Phi$ by requiring that $\Phi$ preserves not only the fibration but also each particular fiber, i.e. the functions $H$ and $F$  (in other words, rearrangements of fibers are temporarily forbidden, i.e.  $\Phi$ induces the identity map on the base of the fibration). 
\end{rem}

The following statement reduces this 4-dim problem for $\Omega$ and $\widetilde\Omega$ to a similar problem for the reduced forms $\omega_\lambda$ and $\widetilde\omega_\lambda$ (in other words, we now reduce our ``two-degrees-of-freedom'' problem to a parametric ``one-degree-of-freedom'' problem).

Consider the singular fibration defined by the functions $H=x^2 + y^3 + \lambda y$ and $F=\lambda$.  This fibration is obviously Lagrangian w.r.t. any of the symplectic structures \eqref{eq:omega} and \eqref{eq:tildeomega} in a neighborhood of the parabolic orbit $\gamma_0=\{x=y=\lambda=0\}$.

\begin{proposition}\label{prop:3.3}
 The following two statements are equivalent.
\begin{enumerate}
\item[(i)] In a tubular neighborhood of the parabolic orbit $\gamma_0$ there is a (real-analytic) diffeomorphism $\Phi$  such that
\begin{itemize}
\item $\Phi$ preserves $H$ and $F$;
\item $\Phi^* (\widetilde\Omega) = {\Omega}$.
\end{itemize}
\item[(ii)] There exists a one-parameter family of local diffeomorphisms $\psi_\lambda (x,y)$ (real-analytic in $x$, $y$ and $\lambda$) leaving fixed the origin in $\R^2(x,y)$ at $\lambda = 0$ and such that, for each $\lambda\in\R$ close enough to $0$,
\begin{itemize}
\item $\psi_\lambda$ preserves $H(x,y,\lambda)$;
\item $\psi_\lambda^* (\widetilde\omega_\lambda) = {\omega}_\lambda$.
\end{itemize}
\end{enumerate}
\end{proposition} 

Roughly speaking, this statement says that the additional terms in \eqref{eq:omega} and \eqref{eq:tildeomega} are not important and can be ignored.
We also remark that we can replace the conditions that $H$ and $F$ are preserved by saying that the fibration is preserved.

\begin{proof} 

The fact that (i) implies (ii) is almost obvious.  Indeed, since $\Omega$ and $\wt\Omega$ are of quite special form, $\Phi^*(\wt\Omega)=\Omega$ and $F=\lambda$ is preserved, then in local coordinates $x,y,\lambda,\varphi$, the diffeomorphism $\Phi$ takes the following form:

$$
\begin{aligned}
\tilde x &= \tilde x (x,y,\lambda), \\ 
\tilde y &= \tilde y (x,y,\lambda),\\
\tilde \lambda &= \lambda, \\
\tilde \varphi &= \varphi + R(x,y,\lambda),
\end{aligned}
$$
then if we consider the first two functions as a family of diffeomorphisms $\psi_\lambda (x,y)$, then we will immediately see that (ii) holds.
Since $\Phi$ preserves $H$ and $F$, it leaves invariant the set of such points $(x,y,\lambda,\varphi)$ that $dH(x,y,\lambda,\varphi)$ and $dF(x,y,\lambda,\varphi)$ are proportional. But for $\lambda = 0$ this set coincides with $\gamma_0$, so $\Phi$ maps $\gamma_0$ to itself. Therefore $\psi_0(0,0)=(0,0)$.

The proof of the converse statement consists of two steps. Assuming that $\psi_\lambda (x,y)$ satisfies the conditions from (ii), we define $\Phi_1$ as follows:
$$
\begin{aligned}
(\tilde x, \tilde y ) &= \psi_\lambda (x,y), \\
\tilde \lambda &= \lambda, \\
\tilde \varphi &= \varphi. 
\end{aligned}
$$
It is easily checked that, for this $\Phi_1$, the symplectic forms $\Phi_1^*(\wt\Omega)$ and $\Omega$ coincide up to additional terms, that is
\begin{equation} \label{eq:difference}
\Phi_1^* (\wt\Omega) - \Omega = d\lambda \wedge \bigl(P(x,y,\lambda) dx + Q(x,y,\lambda) dy\bigr).
\end{equation}
Hence, our goal is to show that these additional terms do not play any essential role and can be killed by an appropriate shift $\varphi \mapsto \varphi - R(x,y,\lambda)$ (without changing the other coordinates). In other words, we need to find $R(x,y,\lambda)$ such that $d\lambda \wedge dR(x,y,\lambda) = d\lambda \wedge \bigl(P(x,y,\lambda) dx + Q(x,y,\lambda) dy\bigr)$. The existence of such a function follows immediately from the closedness of the form  
$d\lambda \wedge \bigl(P(x,y,\lambda) dx + Q(x,y,\lambda) dy\bigr)$  (this form is the difference of two closed forms $\Phi_1^* (\wt\Omega)$ and $\Omega$. Finally, we define $\Phi$ as the composition of $\Phi_1$ and the above shift, 
and we get $\Phi^* (\wt\Omega) = \Phi_1^* (\wt\Omega) - d\lambda \wedge dR(x,y,\lambda) =\Omega$ due to \eqref {eq:difference}.

It remains to notice that, since $\psi_0(0,0)=(0,0)$ and $\gamma_0=\{x=y=\lambda=0\}$, we have $\Phi(\gamma_0)=\gamma_0$, thus $\Phi$ is defined in a neighborhood of $\gamma_0$ as required.
\end{proof}

Our next observation is that symplectic invariants do exist, in other words,   the desired map $\Phi$ (or, equivalently, the family $\psi_\lambda$) may not exist.  Moreover, the existence of just one map $\psi_0$ implies rather strong condition.  To show this, we treat the case $\lambda = 0$ in detail.



\section{The case $\lambda=0$,  one-degree of freedom problem} \label{sec:4}

In this Section, for notational convenience, we use a different sign in the definition of $H$.  Consider the function $H= y^3-x^2$  (in a neighborhood of the origin) and two symplectic forms $\omega_0$ and $\wt \omega_0$ (all of our objects are real-analytic).   We want to know necessary and sufficient conditions for the existence of a local diffeomorphism $\psi_0$   satisfying $\psi_0^* \wt \omega_0 = {\omega}_0$ and   (two versions):
\begin{itemize}
\item either preserving $H$  (strong condition);
\item or preserving the (singular) fibration defined by $H$ (weaker condition) (more formally,   $\psi_0^* (H) = h(H)$  where $h(H)$ is real-analytic and $h'(0)\ne 0$).
\end{itemize}

\begin{figure}[htbp]
\begin{center}
\includegraphics[width=0.40\textwidth]{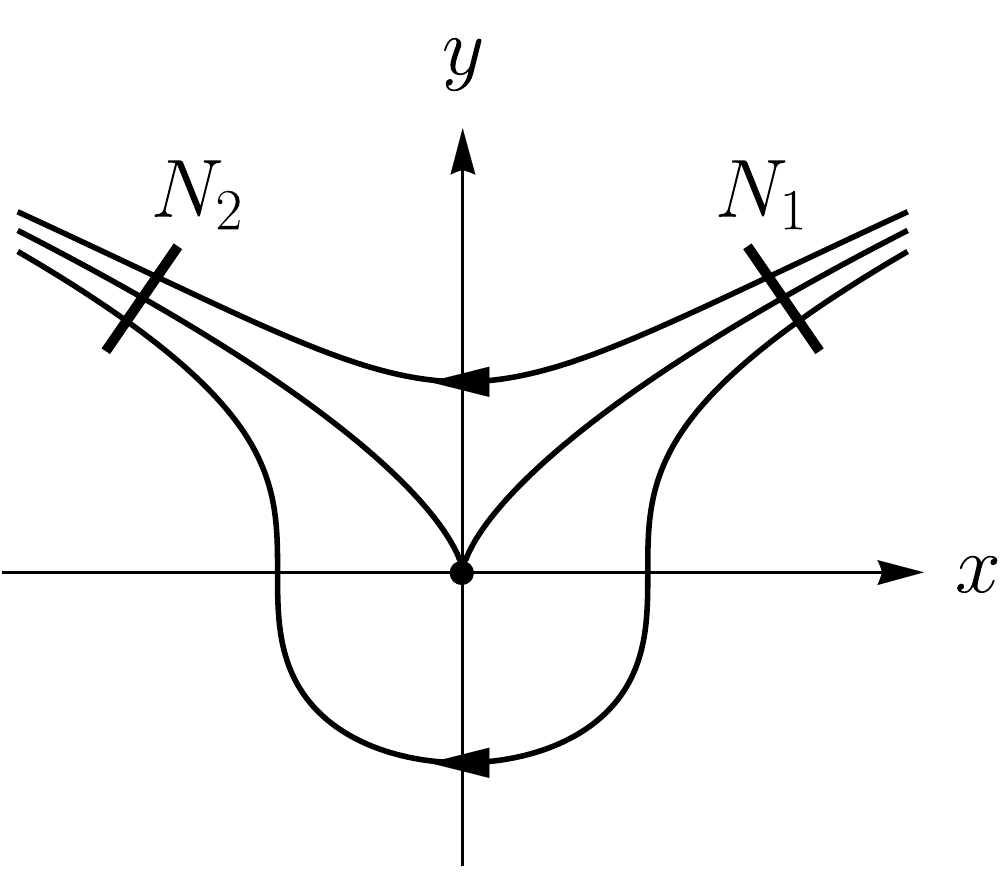}
\end{center}
\caption{Two cross-sections $N_1,N_2$ to the fibration defined by $H=y^3-x^2$}
\label {fig:1}
\end{figure}

The complex version of the first problem was studied in \cite{Francoise}, in this Section we adapt some of these results to the real case we are considering. In the following, $\mathbb R\{H\}$ and $\mathbb C\{H\}$ will denote, respectively, real-analytic germs and complex-analytic germs in the variable $H$ at $0$, i.e., convergent power series in the respective fields. Consider $H = y^3-x^2$ as a holomorphic function, we have:

\begin{proposition}[{\cite[Theorems 2.3 and 3.0]{Francoise}}] 
Any holomorphic $2$-form $\omega_0$  can be decomposed, in a sufficiently small neighborhood $U$ of $0 \in \mathbb{C}^2$, as follows
\begin{equation}\label{decomposition}
\omega_0 = \alpha(H) dx \wedge dy + \beta(H) y dx \wedge dy + d H \wedge d \eta
\end{equation}
for some holomorphic germ $\eta(x,y)$, and unique $\alpha, \beta\in \mathbb{C}\{H\}$.
\end{proposition}

\begin{rem} 
If $\omega_0$ is symplectic, then $\alpha(0) \neq 0$.
\end{rem}

In our case we are dealing with real objects $\omega_0$ and $H$, in this case  $\alpha(H), \beta(H)$ are  real-analytic, and $\eta(x,y)$ can be chosen to be real-analytic, as can be shown by taking the real part of Equation \eqref{decomposition}.

Choose two one-dimensional cross-sections $N_1, N_2$ to the fibration defined by $H$ as shown in Fig.1.  Each non-singular leaf $\tau_H$ of this fibration (with a given value of $H$) now will be interpreted as a trajectory of the Hamiltonian vector field $X_H = \omega_0^{-1}(dH)$  with respect to the symplectic form $\omega_0$.  For each trajectory $\tau_H$ we can measure the \emph{passage time}  $\Pi(H)$ from $N_1$ to $N_2$.  This function can be expressed as
\begin{equation}\label{eqPiLeray}
\Pi(H) = \int_{N_1}^{N_2}\frac{\omega_0}{d H}
\end{equation}
(integral taken along the trajectory $\tau_H$) where $\omega_0/dH$ is the \emph{ Gelfand-Leray form} associated to the pair $(\omega_0,H)$, i.e., any $1$-form $\gamma$ defined in the region $dH \neq 0$ and such that $dH\wedge \gamma= \omega_0$ (the form $\gamma$ is not uniquely defined, but its restriction to the level-sets $H= \mathrm{const}$ is unique).

We can similarly consider the \emph{area function} $\mathsf{area}(H)$ defined as the integral of  $\omega_0$ over the subset  
of $\{0 \leq H(x,y) \leq H\}$ bounded by the sections $N_1$, $N_2$.  As a consequence of Fubini's theorem,
one has
\begin{equation}\label{darea}
\frac{d \mathsf{area}(H)}{d H} =\Pi(H).
\end{equation}

Clearly,  $\Pi(H)$ is a real-analytic function defined for all  (small) $H$.   As $H$ tends to $0$, the passage time $\Pi(H)$ tends to infinity and it is natural to look at the asymptotic behaviour of $\Pi(H)$ at zero.

\begin{lemma}\label{periodsExp}
The function $\Pi(H)$ for $H>0$ can be written as
\begin{equation}
\label{eq:Pi}
\Pi(H) = a(H) H^{-1/6} + b(H) H^{1/6} + c(H), \quad  H>0,
\end{equation}
where $a, b, c \in \mathbb R\{H\}$. 
Moreover,
$a (H)= C_0 \alpha(H)$ and $b(H) = C_1 \beta(H)$ for some non-zero constants $C_0, C_1 \in \mathbb R$, with $C_0> 0$ and $C_1< 0$.
\end{lemma}

Before proving the lemma, we give some remarks:

\begin{itemize}

\item The functions in this representations are uniquely defined, i.e., if
$$
a(H) H^{-1/6} + b(H) H^{1/6} + c(H) = \tilde a(H) H^{-1/6} + \tilde b(H) H^{1/6} + \tilde c(H),
$$
then $a(H)=\tilde a(H)$,  $b(H)=\tilde b(H)$ and $c(H)=\tilde c(H)$.

\item If we change the sections $N_1$ and $N_2$ by a deformation in the class of such sections, then the function $\Pi(H)$ changes by adding a certain analytic function, given by the passage time between the old and the new sections. However, if we replace $N_1$ and $N_2$ by each other, then the function $\Pi(H)$ will be replaced by $-\Pi(H)$. This shows that the functions $a(H)$ and $b(H)$ (up to multiplying with $-1$ simultaneously) do not depend on the choice of the cross-sections $N_1$ and $N_2$. Since we are working with a symplectic form, we have $\alpha(0)\ne0$ and $a(0) \neq 0$, and we can be more specific: the functions $a(H)$ and $b(H)$ with $a(0)>0$ do not depend on the choice of the cross-sections $N_1$ and $N_2$.

\item In a similar way,  we can define the functions $\tilde a$, $\tilde b$ and $\tilde c$ for the second symplectic structure $\widetilde{\omega}_0$.   If $\psi$ preserves $H$ and transforms $\omega_0$ to $\widetilde{\omega}_0$,   then the Hamiltonian vector field $X_H$ will be transformed to the Hamiltonian vector field $\widetilde{X}_H = \widetilde{\omega}_0^{-1}(dH)$  (with the same Hamiltonian $H$).   Since $\psi$ does not preserve the cross-sections $N_1$ and $N_2$, the passage time $\widetilde{\Pi}(H)$  will, in general, differ from $\Pi (H)$ by adding some analytic functions (and, possibly, by multiplying with $-1$),  which shows that the functions $a(H)$ and $b(H)$  with $a(0)>0$ remain invariant under $\psi$, i.e. $a(H)=\tilde a(H)$,  $b(H)=\tilde b(H)$, provided that $\tilde a(0)>0$ too. In other words, $a(H)$ and $b(H)$ with $a(0)>0$ are symplectic invariants  (under the condition that $\psi$ preserves $H$).

\item It is easy to give an example of two symplectic structures producing two different pairs of functions $a$ and $b$ in the asymptotic decomposition \eqref{eq:Pi}.

\end{itemize}

\begin{proof}[Proof of Lemma \ref{periodsExp}] 
Consider the decomposition \eqref{decomposition}. Taking the integral of the Gelfand-Leray form we get:
$$\Pi(H) = \alpha(H)\int_{N_1}^{N_2} \frac{dx \wedge dy}{d H} + \beta(H) \int_{N_1}^{N_2}\frac{y dx \wedge dy}{dH} +  N_2^*\eta(H)-N_1^*\eta(H)$$
(the coefficients can be taken outside of the integral, since we integrate along a trajectory $\tau_H$ where $H$ is constant). The last two terms give a real-analytic contribution. To finish the proof it is sufficient to show that, for $H>0$
\begin{equation}\label{basicInt}
\int_{N_1}^{N_2} \frac{dx \wedge dy}{d H} - C_0 H^{-1/6} \in \mathbb{R}\{H\}, \qquad 
\int_{N_1}^{N_2} \frac{y dx \wedge dy}{d H} -  C_1 H^{1/6} \in \mathbb{R}\{H\},
\end{equation}
for some non-zero real constants $C_0, C_1$, so  that $a(H) = C_0\alpha(H)$ and $b(H)= C_1\beta(H)$. We can assume that $N_1 = \{x = 1\}$ and $N_2 = \{x = -1\}$. We have
$$
\frac{y^j dx\wedge dy}{dH} = -\frac{dx}{3 y^{2-j}},\quad j = 0,1.
$$
Hence, we are reduced to compute, for $j = 0,1$, the integral:
\begin{align*}
J_{j}(H)&=-\frac{1}{3}\int_{1}^{-1}y(H,x)^{j-2}\,dx = \frac{2}{3}\int_{0}^{1}(H+x^{2})^{\frac{j-2}{3}}\,dx\\
&=\frac{2}{3} H^{\frac{j-2}{3}}\int_{0}^{1}\left(1+\tfrac{x^2}{H}\right)^{\frac{j-2}{3}}\,dx 
=\frac{1}{3}H^{\frac{j-2}{3}}\int_{0}^{1}t^{\tfrac{1}{2}-1} \left(1+\tfrac{t}{H}\right)^{\frac{j-2}{3}}\,dt\\
&=\frac{2}{3}H^{\frac{j-2}{3}}F\left(\tfrac{2-j}{3},\tfrac{1}{2},\tfrac{3}{2}; -\tfrac{1}{H}\right)
\end{align*}
where $F(p,q,r;z)$ is the hypergeometric function. In this case we can use the connection formula (\cite[Eq.~(9.5.9)]{Lebedev})
\begin{align*}
F(p,q,r;z) &= \quad c_1(-z)^{-p}F(p, 1+p-r,1+p-q; 1/z)+\\
&\quad+c_2(-z)^{-q}F(q, 1+q-r,1+q-p; 1/z)
\end{align*}
where 
$$
c_1 = \frac{\Gamma(r)\Gamma(q-p)}{\Gamma(r-p)\Gamma(q)},\quad  
c_2 = \frac{\Gamma(r)\Gamma(p-q)}{\Gamma(r-q)\Gamma(p)}.
$$

This gives:
\begin{align*}
J_j(H) &= \tfrac{2}{3}H^{\frac{j-2}{3}}\Big(c_1H^{\frac{2-j}{3}}F(\tfrac{2-j}{3},\tfrac{1-2j}{6},\tfrac{7-2j}{6};-H)+c_2H^{1/2}F(\tfrac{1}{2},0,\tfrac{5+2j}{6};-H)\Big)\\
&= \tfrac{2}{3}c_1 F(\tfrac{2-j}{3},\tfrac{1-2j}{6},\tfrac{7-2j}{6};-H) + \tfrac{2}{3}c_2H^{\frac{2j-1}{6}}\\
&= C_j H^{\frac{2j-1}{6}} + d_j(H), \qquad d_j\in\mathbb R\{H\}, 
\end{align*}
where $C_0 = \frac{\sqrt\pi}{3} \frac{\Gamma(1/6)}{\Gamma(2/3)}$ and $C_1 = \frac{\sqrt\pi}{3}\frac{\Gamma(-1/6)}{\Gamma(1/3)}$. This proves \eqref{basicInt} as required.
\end{proof}

For $r \in \mathbb Q$, consider the operator $\phi_r: \mathbb R\{H\} \rightarrow \mathbb R\{H\}$ defined by $\phi_r: A(H) \mapsto A'(H) H + rA(H)$. If $r \notin \mathbb Z$ then $\phi_r$ is bijective.

\begin{cor}\label{areaExp}
The function $\mathsf{area}(H)$ for $H\ge0$ can be written as
\begin{equation}
\label{eq:area}
\mathsf{area}(H) = A(H) H^{5/6} + B(H) H^{7/6} +  C(H), \quad  H \ge 0,
\end{equation}
where $A, B, C\in \mathbb R\{H\}$ are the unique real-analytic germs such that
$$
a(H) = A'(H) H + \tfrac{5}{6}A(H),\quad b(H) =B'(H) H + \tfrac{7}{6}B(H),\quad c(H)=C'(H), \ C(0)=0,
$$
in other words $A= \phi_{5/6}^{-1}(a)$, $B = \phi_{7/6}^{-1}(b)$.  \hfill $\square$
\end{cor}

\begin{thm}\label{Moser}
Let $\omega_0,\wt \omega_0$ be two real-analytic symplectic forms. Suppose that $\omega_0 - \wt\omega_0 = dH \wedge d\eta$ for some real-analytic function germ $\eta(x,y)$ at $0\in \mathbb R^2$. Then there is a local diffeomorphism $\psi$ at $0 \in \mathbb R^2$ such that $\psi^*H = H$ and $\psi^*\wt\omega_0= \omega_0$.
\end{thm}

\begin{proof} 
We adapt the proof of \cite[Theorem 2.1]{Francoise}. Put $\omega_t = \omega_0 + t(\wt\omega_0-\omega_0)$. In a neighborhood of zero, the forms  $\omega_t$, for $t \in [0,1]$, are also non-degenerate, indeed the equation $\omega_0 - \wt\omega_0 = dH \wedge d\eta$ implies that the two forms have the same sign/orientation at zero and near zero. 
Since $\omega_t$ is a convex combination of functions with the same sign, it will also be non-zero in a neighborhood of zero. Define a real-analytic time-dependent vector field $X_t$ by
$$
i_{X_t}\omega_t = -\eta d H.
$$
Let $\phi_t$ be the flow generated by such $X_t$, we can integrate it for $t \in [0,1]$. Notice that
$$
L_{X_t}\omega_t = i_{X_t}{d\omega_t } + d i_{X_t}\omega_t = d H\wedge d\eta = \omega_0-\wt\omega_0,
$$
therefore
$$
\frac{d}{d t} \phi_t^*\omega_t = \phi^*_t\left(L_{X_t}\omega_t + \frac{d}{d t} \omega_t\right) = 0,
$$
so that $\phi_1^*\wt\omega_0 = \omega_0$. Moreover $L_{X_t}H =0$, because of the equality:
\begin{equation}
0 = i_{X_t}(d H\wedge\omega_t) = (i_{X_t}d H)\omega_t+ d H\wedge i_{X_t}\omega_t = (L_{X_t}H)\omega_t+d H\wedge i_{X_t}\omega_t= (L_{X_t}H)\omega_t,
\end{equation}
and using  $\omega_t \neq 0$. This means that $H\circ \phi_t = H$. Finally take $\psi = \phi_1$.
\end{proof}

In the following we will specify as a subscript the symplectic structure $\omega_0$ in the notation for $\alpha,\beta$, $a,b$ and $A,B$, that is, writing $\alpha_{\omega_0},\beta_{\omega_0}$, $a_{\omega_0},b_{\omega_0}$ and $A_{\omega_0},B_{\omega_0}$. In the rest of the Section, for the reasons explained in the second and third remarks below Lemma \ref{periodsExp}, we will consider symplectic forms inducing a fixed orientation. In this regard we can consider, without loss of generality, only  symplectic forms $\omega$ satisfying  $\alpha_\omega(0) >0$.  Such a symplectic form is said to be \emph{positively-oriented}.

 In the above setting and notation we come to the following statement:

\begin{proposition}\label{Hpreserving}
Let $\omega_0, \wt \omega_0$ be positively-oriented symplectic forms. An $H$-preserving map $\psi$  such that $\psi^*\wt \omega_0={\omega}_0$ exists, if and only if the following conditions hold:
$$
\alpha_{\omega_0}(H) =  \alpha_{\wt \omega_0}(H) \mbox{ and } \beta_{\omega_0}(H)= \beta_{\wt \omega_0}(H)
$$
or, equivalently, $a_{\omega_0}(H) = a_{\wt \omega_0}(H) \mbox{ and } b_{\omega_0}(H)= b_{\wt \omega_0}(H) $ or   $A_{\omega_0}(H) =  A_{\wt \omega_0}(H) \mbox{ and } B_{\omega_0}(H)= B_{\wt \omega_0}(H)$.
\end{proposition}

\begin{proof} Sufficiency: suppose $\alpha_{\omega_0} = \alpha_{\wt \omega_0}$ and $\beta_{\omega_0} = \beta_{\wt \omega_0}$, then $\omega_0 - \wt \omega_0 = dH \wedge d\eta$ for some real-analytic germ $\eta$, and Theorem \ref{Moser} proves the assertion. From Lemma \ref{periodsExp} and Corollary \ref{areaExp} we know that equalities of any of these invariants  are equivalent.

Necessity: suppose $\psi$ exists, let us prove that the invariants coincide. Since $\psi$ preserves $H$ and sends $\wt \omega_0$ to $\omega_0$, we conclude (due to Lemma \ref {periodsExp} and the third remark below it) that $a_{\omega_0}(H) = a_{\wt \omega_0}(H)$ and $b_{\omega_0}(H) =  b_{\wt \omega_0}(H)$, which implies $\alpha_{\omega_0}(H) = \alpha_{\wt \omega_0}(H)$ and $\beta_{\omega_0}(H) = \beta_{\wt \omega_0}(H)$.
\end{proof}

It follows from the above proposition, together with the first remark below Lemma \ref{periodsExp}, that:

\begin{cor}\label{samePI}
Let $\omega_0, \wt \omega_0$ be positively-oriented symplectic forms.
An $H$-preserving map $\psi$  such that $\psi^*\wt \omega_0={\omega}_0$ exists if and only if $\Pi_{\omega_0}(H) - \Pi_{\wt\omega_0}(H)$, which is defined on $\{H>0\}$, extends to a real-analytic function in a neighborhood of $H=0$.\hfill$\square$
\end{cor}

We can also reformulate this result in terms of {\it normal forms}.

\begin{proposition}
For $H= y^3-x^2$ and $\omega_0=f(x,y) dx\wedge dy$ there is a real-analytic local coordinate system $u, v$ and germs $\alpha,\beta\in\mathbb R\{H\}$ such that 
$$
H=v^3 -u^2\quad \mbox{and} \quad  \omega_0=  \alpha(H) \cdot du\wedge dv + \beta(H) \cdot v \,du\wedge dv.
$$  
For positively-oriented symplectic forms, the functions $\alpha(H)$ and $\beta(H)$ are uniquely defined   (the coordinates $u, v$ are not).\hfill$\square$
\end{proposition}

Let us now see what happens if $\psi$ does not preserve $H$, but transforms it to a function of the form $h(H)$, $h'(0)\ne 0$  (in fact $h'(0)>0$). Let $\omega_0, \wt \omega_0$ be positively-oriented symplectic forms. We consider necessary and sufficient conditions for the existence of a local diffeomorphism $\psi$ such that $\psi^*\wt \omega_0 = \omega_0$ and $\psi^*H = h ( H)$ with $h'(0) >0$,
i.e.\ a local symplectomorphism $\psi$ making the following diagram commutative:
$$
\begin{tikzcd}
  (\R^2, 0) \arrow{r}[]{\psi}  \arrow{d}{H}&  (\R^2, 0)    \arrow{d}[]{H}\\
(\R, 0) \arrow{r}{h}& (\R, 0).
\end{tikzcd}
$$

\begin{lemma}\label{relations} 
Suppose there exists $\psi$ such that $\psi^*\wt \omega_0 = \omega_0$ and $\psi^*H = h(H)$ with $h(H) = H\cdot g(H)$, $g(0)>0$. Then we have the following relations:
\begin{itemize}
\item[i)]
$$\begin{cases}A_{\omega_0}(H) = g(H)^{5/6}A_{\wt \omega_0}(h(H)),\\ 
B_{\omega_0}(H) = g(H)^{7/6}B_{\wt \omega_0}(h(H)),\end{cases}$$
\item[ii)] $$\begin{cases}\alpha_{\omega_0}(H) =  g(H)^{-1/6}\left( g'(H) H + g(H)\right) \alpha_{\wt \omega_0}(h(H)),\\ 
\beta_{\omega_0}(H) =  g(H)^{1/6}\left( g'(H) H + g(H)\right)\beta_{\wt \omega_0}(h(H)),\end{cases}$$
\item[iii)] $$\begin{cases}a_{\omega_0}(H) =  g(H)^{-1/6}\left( g'(H) H + g(H)\right) a_{\wt \omega_0}(h(H)),\\ 
b_{\omega_0}(H) =  g(H)^{1/6}\left( g'(H) H + g(H)\right)b_{\wt \omega_0}(h(H)).
\end{cases}$$
 \end{itemize}
\end{lemma}

For proving Lemma \ref {relations}, we need the following:

\begin{lemma}\label{lemmarh}
For any real-analytic map $h(H)$ with $h(0)=0$ and $h'(0)>0$ 
there exists $\psi$  (local real-analytic diffeomorphism) such that 
$$
H(\psi (x,y)) = h (H(x,y)).
$$
In other words, any local diffeomorphism germ $H \mapsto h(H)$ at $0$ with $h'(0)>0$ is liftable.
\end{lemma}

\begin{proof} 
Let $h(H) = H\cdot g(H)$. Define $r_h(x,y):= \left( g(H(x,y))^{1/2}x, g(H(x,y))^{1/3}y\right)$, then
$$
H(r_h(x,y)) = g(H(x,y)) H(x,y) = h(H(x,y)).\eqno\qedhere
$$
\end{proof}

\begin{proof}[Proof of Lemma \ref{relations}]
i) The integrals below are taken over subsets bounded by the sections $N_1$, $N_2$. For any $H\ge0$, we have
\begin{align*}
\mathsf{area}_{\omega_0}(H) &= \int_{0 \leq H(x,y) \leq H} \omega_0 =  \int_{0 \leq H(x,y) \leq H} \psi^*\wt\omega_0 \\
&= \int_{\psi(0 \leq H(x,y) \leq H)} \wt\omega_0  + D(H) =  \int_{0 \leq H(x',y') \leq h(H)} \wt\omega_0  + D(H) \\
&= \mathsf{area}_{\wt \omega_0}(h(H))+  D(H), \qquad D\in
\mathbb{R}\{H\}.
\end{align*}
Substituting in \eqref{eq:area} and comparing coefficients gives (i).

ii) Put $\Psi := \psi \circ r_h^{-1}$, then $\Psi^*\wt \omega_0 = (r_h^{-1})^*\omega_0$ and $H\circ \Psi = H$. Therefore by Proposition \ref{Hpreserving} we have
$$
\wt \omega_0 - (r_h^{-1})^*\omega_0 = dH \wedge d\eta
$$
or again
$$
r_h^*\wt \omega_0 - \omega_0 = dH\wedge d\left( 
\frac{dh(H)}{dH} r_h^*  \eta\right).
$$

We want to understand the relationship between $\alpha_{\omega_0},\beta_{\omega_0}$ and $\alpha_{\wt \omega_0},\beta_{\wt \omega_0}$. The Jacobian matrix of the transformation $r_h$ is given by:
$$
J_{r_h} = \begin{pmatrix}
\frac{1}{2}g(H)^{-1/2}g'(H)\frac{\partial H}{\partial x} x  + g(H)^{1/2}  & \frac{1}{2}g(H)^{-1/2}g'(H)\frac{\partial H}{\partial y} x \\
\frac{1}{3}g(H)^{-2/3}g'(H)\frac{\partial H}{\partial x} y & \frac{1}{3}g(H)^{-2/3}g'(H)\frac{\partial H}{\partial y} y + g(H)^{1/3}
\end{pmatrix}
$$
and 
$$
|J_{r_h}| = \det J_{r_h} = g(H)^{-1/6}\left( g'(H) H + g(H)\right).
$$
This means that:
\begin{align*}
r_h^* \wt \omega_0 &=\alpha_{\wt \omega_0}(h(H)) |J_{r_h}|dx \wedge dy + \beta_{\wt \omega_0}(h(H))| J_{r_h}|g(H)^{1/3}ydx \wedge dy + dh(H) \wedge d(r_h^* \wt \eta).
\end{align*}
The last term can be written as $dH \wedge d \eta'$ with $\eta' = \frac{d h(H)}{dH} r_h^* \wt \eta$. By uniqueness of the characteristic series we get the relations.

\medskip
iii) Follows from (ii) and Lemma \ref{periodsExp}.
\end{proof}

Let's discuss the opposite statement:

\begin{proposition} 
Consider two  positively-oriented symplectic forms $\omega_0, \wt \omega_0$. Suppose there exists a real-analytic map $h(H) = H\cdot g(H)$ such that $g(0)>0$ and some of the three relations (i), (ii), (iii) of Lemma \ref{relations} is satisfied. Then there exists an $H$-fibration preserving map $\psi$ such that $\psi^*\wt \omega_0 = \omega_0$.
\end{proposition}

\begin{proof} 
We assume (ii) is satisfied, the other two cases are equivalent. Consider the map $r_h$ from the proof of Lemma \ref{lemmarh}. It satisfies $r_h^* H = h(H)$ and (trivially) transforms  $\wt \omega_0$ to $r_h^*\wt \omega_0$, therefore by Lemma \ref{relations}
\begin{align*}
\alpha_{r_h^*\wt \omega_0}(H) =  g^{-1/6}(H)\left( g'(H) H + g(H)\right)\alpha_{\wt \omega_0}(h(H)) = \alpha_{\omega_0}(H),\\
\beta_{r_h^*\wt \omega_0}(H)=  g^{1/6}(H)\left( g'(H) H + g(H)\right) \beta_{\wt \omega_0}(h(H)) = \beta_{\omega_0}(H).
\end{align*}
This shows that $\omega_0$ and $r_h^*\wt \omega_0$ have the same charateristic series, therefore
$$
\omega_0 - r_h^*\wt \omega_0 = d H\wedge d\eta,
$$
for some real-analytic germ $\eta$. As we know from the case of $H$-preserving maps, this equation implies the existence of a $H$-preserving diffeomorphism $\phi$ such that $\phi^*r_h^*\wt \omega_0 = \omega_0$. In conclusion, $\psi = r_h \circ \phi$ is the map we are looking for.
\end{proof}

Finally we show how one symplectic invariant survives in the case of $H$-fibration-preserving maps. Consider a positively-oriented symplectic form $\omega_0$. Suppose we can solve the equation
\begin{equation}\label{eqf1}
\alpha_{\omega_0}(H)=  K g(H)^{-1/6}\left( g'(H) H + g(H)\right), \qquad K \in \mathbb R,
\end{equation}
for $g(H)$. Then the rescaling map $r_h$, with $h(H) = H\cdot g(H)$, transforms $H$ to $h(H)$ and 
the symplectic form $\wt \omega_0 = (r_h^{-1})^*\omega_0$ to $\omega_0$, where  $\alpha_{\wt \omega_0}(H) = K$. After this, the other characteristic series $\beta_{\wt \omega_0}(H)$ survives as an invariant in the usual sense (of $H$-preserving local diffeomorphisms).

\begin{lemma}
The invariant $a_{\omega_0}(H)$ can be reduced to a constant.
\end{lemma}

\begin{proof}  
By assumption we have $a_{\omega_0}(0) > 0$. It follows from Corollary \ref{areaExp} that $A_{\omega_0}(0) > 0$ as well. Setting $g(H) = A_{\omega_0}(H)^{6/5}$ and $\wt\omega_0=(r_{h}^{-1})^*\omega_0$, we obtain from Lemma \ref {relations} that $A_{\wt \omega_0}(H) = 1$. With this choice of $h(H) = H\cdot g(H)$, we have $A_{\omega_0}(H) = g(H)^{5/6}$, therefore by Corollary \ref {areaExp}
\begin{align*}
a_{\omega_0}(H) &= A_{\omega_0}'(H)H + \tfrac{5}{6}A_{\omega_0}(H) \\
&= \tfrac{5}{6}g(H)^{-1/6}\left(g'(H)H +  g(H)\right),
\end{align*}
so that $a_{\wt \omega_0}(H)$ is constant  (and $\alpha_{\wt \omega_0}(H)$ as well).
\end{proof}

\begin{proposition} \label{prop:3.3.11}
A real-analytic singular Lagrangian fibration with one degree of freedom is symplectomorphic, in a neighborhood of an $A_2$ singularity, to (one of) the following model:
$$
H= y^3-x^2, \quad \omega_0= dx\wedge dy + f(y^3-x^2) \cdot y\, dx\wedge dy.
$$
Or, equivalently, in a neighborhood of an $A_2$ singularity (with one degree of freedom) we can always find local coordinates $x$ and $y$ such that the fibration is defined by the function $H=y^3-x^2$ and the symplectic structure takes the form $ \omega_0= dx\wedge dy + f(y^3-x^2) \cdot y\, dx\wedge dy$.  

In this representation, the real-analytic function $f(H)$ is uniquely defined. \hfill $\square$
\end{proposition}

This proposition says that as a complete symplectic invariant of an $A_2$ singular fibration with one degree of freedom we may consider one (real-analytic) function in one variable.   Since such a fibration appears as a symplectic reduction of the Lagrangian fibration near a parabolic orbit  (for $\lambda=0$), we conclude that parabolic orbits possess non-trivial symplectic invariants and the next section is aimed at describing ``all of them''.


\section{Parametric version} \label{sec:5}

Our next step is a parametric version of the above construction.  We now assume that $H$ depends on $\lambda$ as a parameter:
$$
H(x,y,\lambda)= H_\lambda(x,y)= x^2 + y^3 + \lambda y
$$
and for each value of $\lambda$ we consider a symplectic structure $\omega_\lambda = f(x,y,\lambda) dx\wedge dy$, $f>0$.

We first give necessary and sufficient conditions for the existence of 
a family of maps $\psi_\lambda$ from Proposition \ref{prop:3.3}.

Following the same idea as before, we choose two 2-dimensional sections $\mathcal N_1$ and $\mathcal N_2$ analogous to the above sections $N_1$ and $N_2$ (but now for all values of $\lambda$) and define the passage time 
$$
\Pi (H,\lambda) = \int^{\mathcal N_2(H,\lambda)}_{\mathcal N_1(H,\lambda)} \frac{\omega_\lambda}{dH_\lambda}
$$ 
for each trajectory
with parameters $H$ and $\lambda$, $(H,\lambda)\not\in\Sigma_{\mathrm{hyp}}$, see  Fig. 2.
Also we see that  for each $\lambda<0$ we have a family of closed trajectories also parametrized by $H$ and $\lambda$. Let us denote by $\Pi_\circ (H,\lambda)$ the period of these trajectories\footnote{Alternatively we may compute the area  $\mathsf{area}_\circ(H,\lambda)=2\pi I_\circ (H,\lambda)$ enclosed by such a trajectory.  This function can be understood as the action variable corresponding to this family of closed cycles.  Notice that $\Pi_\circ$ and $I_\circ$ are related by differentiation:  $\Pi_\circ(H,\lambda) =  2\pi \frac{\partial }{\partial H} I_{\circ}(H,\lambda)$, comp. \eqref{darea}.}. We can compute these functions for both forms $\omega_\lambda$ and $\widetilde\omega_\lambda$.   For $\widetilde\omega_\lambda$, we denote them by $\widetilde{\Pi}(H,\lambda)$ and $\widetilde{\Pi}_\circ(H,\lambda)$.

\begin{figure}[htbp]
\begin{center}
 \includegraphics[width=0.50\textwidth]{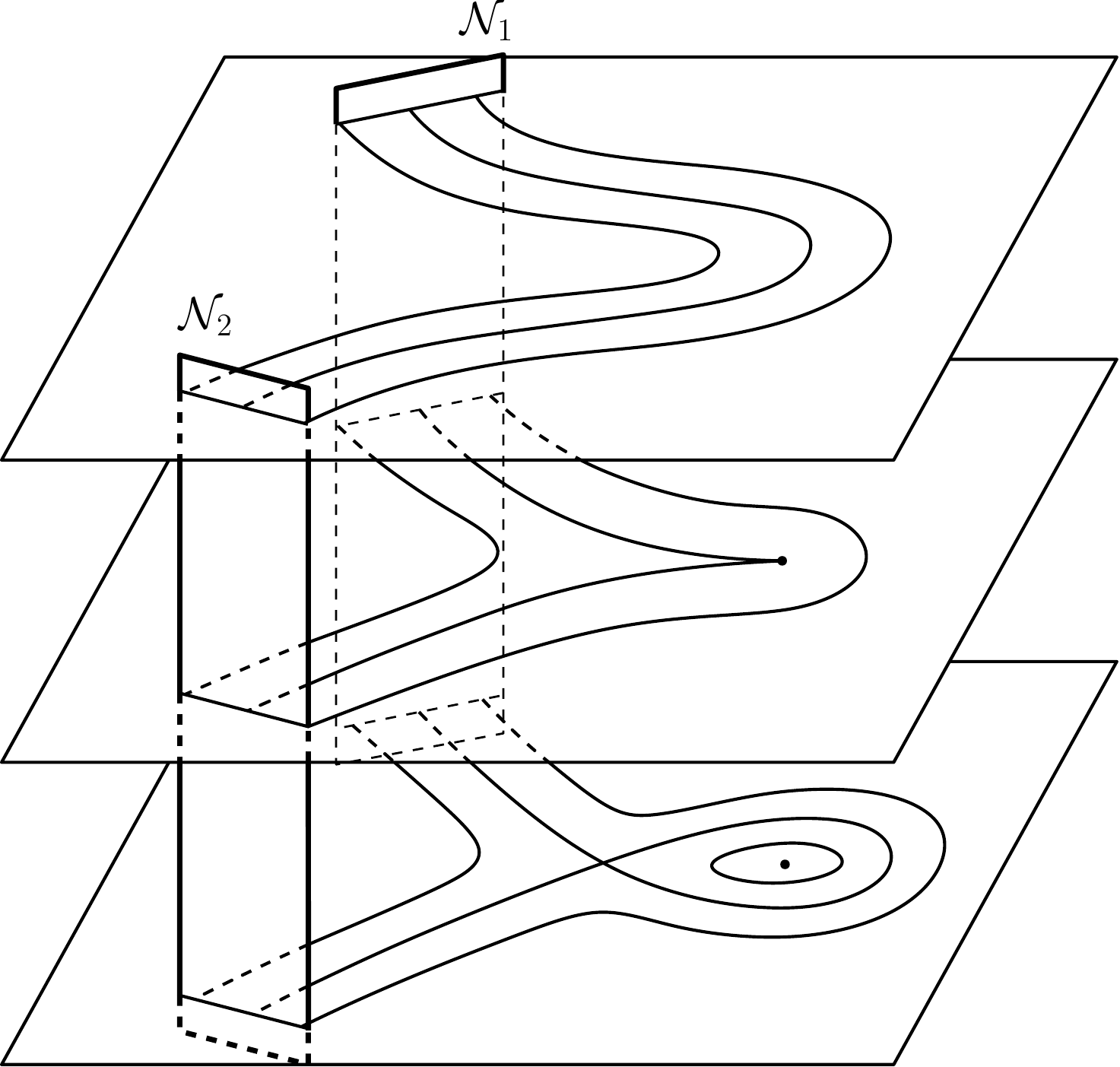}
\end{center}
\caption { Two cross-sections $\mathcal N_1,\mathcal N_2$ to the fibration near a parabolic orbit} \label {fig:2}
\end{figure}

\begin{proposition} \label{prop:5.1}
A family of local diffeomorphisms $\psi_\lambda$ from Proposition \ref{prop:3.3} exists if and only if 
\begin{itemize}
\item[(i)]
$\Pi(H,\lambda)- \widetilde{\Pi}(H,\lambda)$  extends to a real-analytic function in a neighborhood of the point  $H=0, \lambda=0$, 

\item[(ii)] $\Pi_\circ (H,\lambda)= \widetilde{\Pi}_\circ (H,\lambda)$.
\end{itemize}

\end{proposition}

\begin{proof} We need to justify the ``if'' part only.
First of all we notice that, for each $\lambda$  (if we consider each slice $\{\lambda=\mbox{const}\}$ separately),  a map $\psi_\lambda$ exists. 
Indeed, for $\lambda>0$, there are no obstructions for the existence of $\psi_\lambda$ at all,
since our fibration is regular.   For $\lambda =0$, the existence of $\psi_\lambda$ was proved in Corollary \ref{samePI}.
As for $\lambda <0$,  this property follows from non-degeneracy of singular points (see \cite{Dufour}).

We only need to ``combine'' all these maps into one single $\Psi(x,y,\lambda) = \psi_\lambda (x,y)$ in such a way that $\Psi$ is real-analytic with respect to all variables (including $\lambda$).

To that end, we notice first of all that the maps $\psi_\lambda$ can be chosen in such a way that each section $N_{1, \lambda}=\{ (x,y,\lambda)\in \mathcal N_1 \ \mbox{with $\lambda$ fixed}\}$   (i.e. the intersection of $\mathcal N_1$ with the corresponding $\lambda$-slice) is mapped to itself, i.e., $\psi_\lambda|_{N_{1, \lambda}} = \mathrm{id}$.
This choice (of the initial data) makes our construction unique. In other words, we may assume without loss of generality that $\Psi$ leaves $\mathcal N_1$ fixed.  
  
Let $\sigma^t$ and $\widetilde\sigma^t$ denote the Hamiltonian flows of $H_\lambda$ w.r.t.  $\omega_\lambda$ and $\widetilde\omega_\lambda$ respectively. Since $H$ is preserved and 
$\psi_\lambda^*(\widetilde\omega_\lambda)=\omega_\lambda$,
we conclude that $\psi_\lambda$ sends the Hamiltonian flow of $H$ w.r.t $\omega_\lambda$ to that w.r.t. $\widetilde\omega_\lambda$, i.e., the following relation holds
$$
\psi_\lambda \circ \sigma^t = \widetilde\sigma^t \circ \psi_\lambda.
$$
This relation implies a simple ``explicit'' formula for $\psi_\lambda$ (for those points $Q$ which can be obtained from $\mathcal N_1$ by shifting along the flow $\sigma^t$).  Namely,  let  $Q = \sigma^{t(Q)}(Q_0)$ with $Q_0 \in \mathcal N_1$. Then applying the above relation to the point $Q$  with  $t=-t(Q)$ we get
$$
\psi_\lambda \circ \sigma^{-t(Q)} (Q) = \widetilde\sigma^{-t(Q)} \circ \psi_\lambda (Q)\\
$$
or, equivalently,
$$
\psi_\lambda (Q) = \widetilde\sigma^{t(Q)} \circ \psi_\lambda \circ \sigma^{-t(Q)} (Q) 
$$
and, using that $\psi_\lambda \circ \sigma^{-t(Q)} (Q) = \psi_\lambda  (Q_0) = Q_0 =  \sigma^{-t(Q)} (Q)$, we finally get:
\begin{equation}
\label{eq:mainsiplecto}
\psi_\lambda (Q) = 
\widetilde\sigma^{t(Q)} \circ \sigma^{-t(Q)} (Q),
\end{equation}
where the time $t(Q)$ is chosen in such a way that   $\sigma^{-t(Q)} (Q)\in \mathcal N_1$.  Notice that the family $\psi_\lambda$ so defined automatically satisfies the required conditions (ii) from Proposition \ref{prop:3.3}  and is locally analytic w.r.t. all the variables  (including the parameter $\lambda$)  everywhere where it makes sense.  The problem, however, is that  \eqref{eq:mainsiplecto} works neither at the singular points  nor at the points lying on ``small'' closed trajectories that appear for $\lambda <0$ (the reason is obvious: the Hamiltonian flow $\sigma^t$ starting from $\mathcal N_1$ does not reach them). 

Below, we will use a slightly different version of formula \eqref{eq:mainsiplecto}.  Notice that  $Q$ can also be obtained from $Q_0\in \mathcal N_1$ by shifting  along the other Hamiltonian flow $\widetilde{\sigma}^t$, that is, $Q= \widetilde{\sigma}^{\tilde{t}(Q)}(Q_0)$ for some  $\tilde{t}(Q)\in \mathbb R$.  Hence,
$$
\psi_\lambda (Q) = 
\widetilde\sigma^{t(Q)} \circ \sigma^{-t(Q)} (Q) =  \widetilde\sigma^{t(Q)-\tilde t(Q)} \circ \widetilde\sigma^{\tilde t(Q)} \circ \sigma^{-t(Q)} (Q) =
 \widetilde\sigma^{t(Q)-\tilde t(Q)} \circ \widetilde\sigma^{\tilde t(Q)} (Q_0)=  \widetilde\sigma^{t(Q)-\tilde t(Q)} (Q), 
$$
or, finally,
\begin{equation}
\label{eq:mainsiplecto2}
\psi_\lambda (Q) = \widetilde\sigma^{r(Q)} (Q), \qquad \mbox{where $r(Q) = t(Q) - \tilde t(Q)$.}
\end{equation}

This formula has a very natural meaning.  If $Q$ can be obtained by shifting a certain point $Q_0\in \mathcal N_1$  along the flows $\sigma^t$ and $\widetilde \sigma^t$, then $\psi_\lambda$ simply shifts $Q$ along $\widetilde \sigma^t$ by time  $r(Q) = t(Q) - \tilde t(Q)$, where $t(Q)$ (resp.  $\tilde t(Q)$) is the time necessary for the flow  $\sigma^t$ (resp. $\tilde\sigma^t$) to reach $Q$ starting from the section $\mathcal N_1$.

Our goal is to show that this formula extends to a neighborhood of the parabolic point up to a well defined real-analytic map (in the sense of all the variables $x,y$ and $\lambda$).

To that end,  we will use a ``complexification trick''.   Since all the objects under consideration are real-analytic we can naturally complexify them, that is, we may think of $x,y,\lambda$ as complex variables, $H$ and $F$ as complex functions, $\omega$ and $\widetilde \omega$ as complex symplectic forms, etc.  We will also assume that the section $\mathcal N_1$ is given by an analytic equation like  $f(x,y)=0$, so that the same equation defines a (local) complex hypersurface that is transversal to all complexified leaves 
$\mathcal L_{\varepsilon_1,\varepsilon_2}=\{H= \varepsilon_1, F=\varepsilon_2\}$,  $(\varepsilon_1, \varepsilon_2) \subset \mathbb C^2$, for small enough $|\varepsilon_1|+|\varepsilon_2|$.
We are now looking for a local holomorphic map $\Psi(x,y,\lambda)=\psi_\lambda(x,y)$ which preserves $H$ and $F$ and transforms 
$\widetilde\omega_\lambda$ to $\omega_\lambda$.

We want this map to be the ``complexification'' of the family $\psi_\lambda$ defined above  (in particular,  the complex section $\mathcal N_1$ does not move under the action of $\psi_\lambda$).  We keep the same notations for all the objects, but  now we think of them from the
complex viewpoint.  In particular,  the parameter $t$  for the flows $\sigma^t$ and $\widetilde\sigma^t$ is complex and plays the role of ``complex time''.   Similarly,  $t(Q)$, $\tilde t(Q)$ and $r(Q)$ are complex functions which, by construction, are locally holomorphic.

One of the advantages of the complexified picture is that all the leaves 
$\mathcal L_{\varepsilon_1,\varepsilon_2}$ 
(both regular and singular) are now connected, each of them intersects the section $\mathcal N_1$ at exactly one point and, moreover,  every regular point $Q$ (even if it belongs to a singular leaf)  can be joint with $\mathcal N_1$ by a continuous path lying on the leaf.  Notice that the regular part of each leaf $\mathcal L_{\varepsilon_1,\varepsilon_2}$
can be understood as a complex trajectory of the complex flow $\sigma^t$ or $\widetilde\sigma^t$.  

The problem coming with ``complexification'' is that $t(Q)$ and $\tilde t(Q)$ are not uniquely defined anymore.  Indeed, the complex leaf 
$\mathcal L_{\varepsilon_1,\varepsilon_2}=\{H= \varepsilon_1, F=\varepsilon_2\}$
is now a two-dimensional surface with a non-trivial topology.  In particular, the first homology group of 
$\mathcal L_{\varepsilon_1,\varepsilon_2}$
is non-trivial and this leads to the fact that  $Q$ can be reached from $\mathcal N_1$ in many different ways,  e.g.,   $\sigma^{t_1} (Q_0) = \sigma^{t_2}(Q_0) = Q$.  So we need to make sure that the choice of $t_i$ does not affect the final result of (the complex version of) \eqref{eq:mainsiplecto} and \eqref{eq:mainsiplecto2}.  

Let us discuss this issue in more detail.  Consider one particular leaf 
$\mathcal L_{\varepsilon_1,\varepsilon_2}$ (not necessarily regular).
It intersects the section $\mathcal N_1$ at exactly one point $Q_0$.    For 
$Q\in \mathcal L_{\varepsilon_1,\varepsilon_2}$,
consider a path $\gamma(s)$ connecting this point with $Q_0$ so that $\gamma (0)= Q_0$ and $\gamma(1)= Q$.    Each point of this path can be written as $\gamma(s) = \sigma^{t (s)} (Q_0) = \widetilde{\sigma}^{\tilde t (s)} (Q_0) $  with $t(s)\in \mathbb C$, $t(s)$ continuous and $t(0)=0$.  In this way, we set $t(Q)=t(1)$ 
and
$\tilde t(Q) = \tilde t(1)$. It is easy to see that deforming $\gamma(s)$ continuously does not change 
$t(Q)$ and $\tilde t(Q)$.  Thus, this construction shows that $t(Q)$ and $\tilde t(Q)$  (and consequently $r(Q)=t(Q)-\tilde t(Q)$) are uniquely defined if we fix the homotopy type of a curve connecting $Q_0$ and $Q$.  If we choose two homotopically different curves $\gamma_1$ and $\gamma_2$, then,  in general,  $t_1(Q)\ne t_2(Q)$ and $\tilde t_1(Q) \ne \tilde t_2(Q)$.  

The condition we need is  $r_1(Q) = t_1(Q)-\tilde t_1(Q) = t_2(Q)-\tilde t_2(Q) = r_2(Q)$ or, equivalently,  $t_1(Q)-t_2(Q) = \tilde t_1(Q)- \tilde t_2(Q)$.  The latter has a very simple geometric meaning. Namely,  $Q= \sigma^{t_1(Q)} (Q_0) = \sigma^{t_2(Q)} (Q_0)$ means that 
$\sigma^{t_1(Q)-t_2(Q)} (Q)=Q$.  In other words, $t_1(Q)-t_2(Q)$  is the period of 
$\mathcal L_{\varepsilon_1,\varepsilon_2}$
as a ``complex trajectory'' of the flow $\sigma^t$, which corresponds to the (homotopy class of the) 
loop formed by the curves $\gamma_1$ and $-\gamma_2$.  
In other words, the condition that we need can be formulated as follows:  for each 
loop $\gamma$ on $\mathcal L_{\varepsilon_1,\varepsilon_2}$,
the corresponding periods of the Hamiltonian flows generated by $H_\lambda$  w.r.t. the symplectic forms  $\omega_\lambda$ and $\widetilde\omega_\lambda$ coincide.  These periods can be computed explicitly as  (compare with \eqref{eqPiLeray}):
$$
\Pi_\gamma(H,\lambda) = \oint_\gamma \frac{\omega_\lambda}{ dH_\lambda} \quad \mbox{and} \quad \widetilde \Pi_\gamma(H,\lambda) = \oint_\gamma \frac{\widetilde \omega_\lambda}{ dH_\lambda},
$$
so that the required condition takes the following form:
\begin{equation}
\label{eq:maincondition}
\Pi_\gamma(H,\lambda)  =  \widetilde \Pi_\gamma(H,\lambda) 
\end{equation}
for any closed loop $\gamma$ on $\mathcal L_{\varepsilon_1,\varepsilon_2} = \mathcal L_{H,\lambda}$  (equivalently, for any cycle of the first homology group).

Let us assume that this condition holds true  (we will below explain why, under our assumptions, this is indeed the case)  and make the next step of our construction.  As just shown,  \eqref{eq:maincondition} guarantees that  the function $r(Q)$ is well defined for any point $Q$ that can be reached by the flows 
$\sigma^t$ and $\widetilde\sigma^t$ starting from the section $\mathcal N_1$.    Since (after complexification!) every regular point satisfies this property, $r(Q)$ is defined everywhere except for singular points and is locally holomorphic by construction.  But the set of singular points, 
$$
\left\{\frac{\partial H}{\partial x}=0,\ \frac{\partial H}{\partial y}=0\right\}
=\left\{x=0,\ 3y^2+\lambda=0\right\},
$$ 
is an algebraic variety of (complex) codimension 2, and therefore by the second Riemann extension theorem (\cite[Theorem 4.4]{Gunning} or \cite[Theorem 7.2]{GrauertRemmert}), 
$r(Q)$ can be extended up to a holomorphic function defined everywhere in the considered domain.   In particular, this function is bounded and therefore,  by taking a smaller neighborhood of the parabolic point,  we may assume that the flow $\sigma^t$ is well defined for all  $t$ satisfying $|t|\le\max |{r(Q)}|$.

After this, our formula
\eqref{eq:mainsiplecto2} can 
be applied to every point from this neighborhood giving a well defined holomorphic map $\Psi$ with required properties.  It remains to return to the real world  (i.e., restrict $\psi_\lambda$ to the real part of our complex neighborhood) and we are done.

To complete the proof we need to explain why condition \eqref{eq:maincondition} is fulfilled in our case.  First we notice that the first homology group of complex leaves $\mathcal L_{\varepsilon_1,\varepsilon_2}$ is generated by 2 cycles  (topologically, $\mathcal L_{\varepsilon_1,\varepsilon_2}$ is a torus with one hole  if $(\varepsilon_1,\varepsilon_2)\not\in\Sigma^{\mathbb C}=\{\varepsilon_1^2=-\frac{4}{27}\varepsilon_2^3\}$, a 2-disk with one hole if $(\varepsilon_1,\varepsilon_2)=(0,0)$, and a pinched torus with one hole otherwise, where one of the basic cycles is pinched to a point).  Consider the (real) 
``swallow-tail domain'' $\{(\varepsilon_1,\varepsilon_2)\in\R^2\mid \varepsilon_1^2<-\frac{4}{27}\varepsilon_2^3\}\subset\{\lambda <0\}$.  Then one of these two cycles can be chosen real. Such a cycle is shown in Fig. 2 as a small 
loop, whose periods w.r.t.  $\omega_\lambda$ and $\widetilde{\omega}_\lambda$  were denoted by  $\Pi_\circ (H,\lambda)$ and $\widetilde{\Pi}_\circ (H,\lambda)$.  By our assumption $\Pi_\circ (H,\lambda)=\widetilde{\Pi}_\circ (H,\lambda)$, i.e., one of the required conditions  coincides with the second condition (ii) of Proposition  \ref{prop:5.1}.   

Now consider  condition (i)  for $\{(\varepsilon_1,\varepsilon_2)\in\R^2\mid \varepsilon_1^2<-\frac{4}{27}\varepsilon_2^3\}\subset\{\lambda <0\}$.   When approaching a hyperbolic singular leaf, the  functions $\Pi(H,\lambda)$ and $\widetilde \Pi(H,\lambda)$ both have logarithmic singularity.  This is a well known property of non-degenerate hyperbolic points (\cite{Dufour, BolsOsh}),  in other words,  they have the following asymptotics\footnote{In the domain $\{(\varepsilon_1,\varepsilon_2)\in\R^2\mid \varepsilon_1^2>-\frac{4}{27}\varepsilon_2^3\}$, similar asymptotics for $\Pi(H,\lambda)$ and $\wt\Pi(H,\lambda)$ hold, where the coefficients $\alpha,\beta,\wt\alpha,\wt\beta$ are replaced by $2\alpha,\delta,2\wt\alpha,\wt\delta$ for some real-analytic functions $\delta,\wt\delta$ in a neighbourhood of $\Sigma_{\mathrm{hyp}}$.}:
$$
\begin{aligned}
\Pi(H,\lambda) &=  \alpha(H,\lambda) \ln \left| 
3\sqrt{3} H - 2(-\lambda)^{3/2}
\right| + \beta(H,\lambda), \\
\widetilde\Pi(H,\lambda) &=  \widetilde\alpha(H,\lambda) \ln \left| 
3\sqrt{3} H - 2(-\lambda)^{3/2}
\right| + \widetilde\beta(H,\lambda)
\end{aligned}
$$
for some real-analytic functions $\alpha, \beta, \widetilde\alpha, \widetilde\beta$ in a neighborhood of $\Sigma_{\mathrm{hyp}}=\Sigma\cap\{H>0\}$. Condition (i) of  Proposition   \ref{prop:5.1}, therefore, implies that $\alpha(H,\lambda) = \widetilde\alpha(H,\lambda)$.  For  hyperbolic points, this coefficient in front of logarithm is known to be proportional to the period of the second  (invisible in the real setting) cycle on the complex leaf $\mathcal L_{H,\lambda}$  (see Proposition \ref{prop:ap2} and discussion in Appendix). 

Thus, for real $\lambda <0$ and real $H \in\left(-2(-\lambda)^{3/2}/(3\sqrt3),2(-\lambda)^{3/2}/(3\sqrt3)\right)$ the required conditions \eqref{eq:maincondition} are fulfilled.  Since the periods $\Pi_\gamma$ and $\widetilde\Pi_\gamma$ are locally holomorphic (we cannot consider them as single-valued functions because of the monodromy phenomenon) and coincide on an open real domain, we conclude that \eqref{eq:maincondition} is fulfilled identically, which completes the proof of Proposition \ref{prop:5.1}.
\end{proof}

We now return to our discussion on symplectic invariants of parabolic trajectories that we started in Section \ref{sec:3}.   According to Proposition 
\ref{pro:coord},  this problem can be reduced to the situation explained in Remark \ref{rem:about2forms}. 

Namely, we consider two functions $H=x^2+y^3+\lambda y$ and $F=y$ that commute simultaneously with respect to two symplectic forms $\Omega$ and $\wt\Omega$ defined by \eqref{eq:omega} and \eqref{eq:tildeomega}  with $\omega_\lambda = f(x,y,\lambda) dx\wedge dy$ and  $\tilde\omega_\lambda = \tilde f(x,y,\lambda) dx\wedge dy$ and $f, \tilde f >0$. Combining Proposition \ref{prop:3.3} and Proposition \ref{prop:5.1}, we obtain the following 

\begin{proposition}\label{prop:5.2}
The following two statements are equivalent.
\begin{enumerate}
\item In a tubular neighborhood of the parabolic orbit $\gamma_0(t)=(0,0,0, \varphi{=}t)$ there is a (real-analytic) diffeomorphism $\Phi$  such that

\begin{itemize}
\item[(i)] $\Phi$ preserves $H$ and $F$;
\item[(ii)] $\Phi^* (\widetilde\Omega) = \Omega$.
\end{itemize}

\item 
The functions $\Pi,\Pi_\circ,\wt\Pi,\wt\Pi_\circ$ (real-analytic in the complement of the bifurcation diagram) satisfy the relations
\begin{itemize}
\item
$\Pi(H,\lambda)- \widetilde{\Pi}(H,\lambda)$  is real-analytic    (in a neighborhood of the point  $H=0, \lambda=0$), 

\item $\Pi_\circ (H,\lambda)= \widetilde{\Pi}_\circ (H,\lambda)$.
\hfill $\square$
\end{itemize}
\end{enumerate}
\end{proposition}

In fact, the functions $\Pi(H,\lambda)$ and $\Pi_\circ (H,\lambda)$ are not independent.  Indeed, as  
$H\to  2(-\lambda)^{3/2}/(3\sqrt3)$ 
(i.e., when the real disconnected regular fiber approaches the hyperbolic singular one) these two functions  have a logarithmic singularity with the same logarithmic coefficient, that is, we have the following asymptotics:
$$
\begin{aligned}
\Pi(H,\lambda) &=  \alpha(H,\lambda) \ln \left| 
3\sqrt{3} H - 2(-\lambda)^{3/2}
\right| + \beta(H,\lambda), \\
\Pi_\circ(H,\lambda) &=  \alpha(H,\lambda) \ln \left| 
3\sqrt{3} H - 2(-\lambda)^{3/2}
\right| + \beta_\circ(H,\lambda).
\end{aligned}
$$  
In other words, the functions $\beta(H,\lambda)$ and $\beta_\circ(H,\lambda)$ are different and not related to each other in any sense,  but 
the coefficients $\alpha(H,\lambda)$ are the same for the both functions.   According to Proposition \ref{prop:5.2}, however,   the regular part  $\beta(H,\lambda)$ of $\Pi(H,\lambda)$ does not play any role, so that the only important information for us is the coefficient $ \alpha(H,\lambda)$ which, as we have just explained,  can be ``obtained'' from $\Pi_\circ(H,\lambda)$.  Hence we conclude that   
$\Pi_\circ(H,\lambda)$ contains all the information we need for symplectic characterisation of a parabolic trajectory.  

We also note  that the period $\Pi_\circ(H,\lambda)$ of closed trajectories can naturally be interpreted in terms of the action variables of our integrable system. Indeed, the family of {\it small closed trajectories} shown on Fig.~2 corresponds to a family of 
``narrow''
two-dimensional Liouville tori  (recall that a four-dimensional neighborhood $U(\gamma_0)$ of the parabolic orbit $\gamma_0$ is the product (Fig. 2)$\times S^1$). For this family, we can naturally define two action variables $I_1$ and $I_2$. The first of them corresponds to 
the free Hamiltonian $S^1$-action on $U(\gamma_0)$ generated by $F=\lambda$, that is, $I_1 = \lambda$. The other $I_2(H,\lambda)$  corresponds to the family of vanishing cycles shown in Fig.~2  as {\it small closed trajectories}. 
We re-denote this function as $I_2(H,\lambda)=I_\circ (H,\lambda)$.
Without loss of generality  we will assume  that
\begin{equation} \label {eq:action:narrow}
I_\circ>0 \quad \mbox{and} \quad I_\circ\to0 \quad \mbox{as} \quad (H,\lambda)\to (H(\gamma_0),F(\gamma_0)),
\end{equation}
i.e., as we approach the singular fiber.  Notice that,  in a coordinate system $(x,y,\lambda,\varphi)$, $I_\circ (H,\lambda)$ can be defined  by an explicit formula.  
Fixing $H$ and $\lambda$, we define a unique closed cycle (see Fig.~2). This cycle bounds a certain domain $V_{H,\lambda}\subset \R^2 (x,y)$ on the corresponding layer $\{ \lambda = \mathrm{const}\}$.  Then
$$
I_\circ (H,\lambda) = \frac{1}{2\pi} \mathsf{area}_\circ \bigl(V_{H,\lambda}\bigr) =   \frac{1}{2\pi}\int_{V_{H,\lambda}}  \omega_\lambda.
$$

It is well-known that $I_\circ (H,\lambda)$ and $\Pi_\circ(H,\lambda)>0$ are related in the following very simple way:
$$
\Pi_\circ(H,\lambda) = 2\pi  \frac{\partial}{\partial H}I_\circ (H,\lambda),
$$
which shows that $\Pi_\circ(H,\lambda)$ can be reconstructed from $I_\circ (H,\lambda)$, so that we finally come to the following equivalent  version of 
Proposition \ref{prop:5.2}.

\begin{proposition} \label{prop:5.3} 
In the same assumptions as in Proposition \ref{prop:5.2}, the following two statements are equivalent.
\begin{enumerate}
\item[(i)] In a tubular neighborhood of the parabolic orbit $\gamma_0$ there is a (real-analytic) diffeomorphism $\Phi$  such that
\begin{itemize}
\item $\Phi$ preserves $H$ and $F$;
\item $\Phi^* (\widetilde \Omega) = \Omega$.
\end{itemize}
\item[(ii)]  The actions (real-analytic on the 
``swallow-tail domain'' corresponding to a family of ``narrow'' Liouville tori)
corresponding to the family of vanishing cycles 
(cf. \eqref{eq:action:narrow})
coincide, $I_\circ(H,F)= \widetilde I_\circ (H,F)$.   
\hfill $\square$
\end{enumerate}
\end{proposition}

We now want to give one more version of the criterion for the existence of $\Phi$ by omitting the condition $F=\lambda$ which, in particular, means that $F$ is a $2\pi$-periodic integral (equivalently, the action variable $I_1$) simultaneously for the both integrable systems.

Consider  $H$ and $F$ commuting with respect to $\Omega$ and $\widetilde\Omega$ in a tubular neighborhood of a parabolic orbit $\gamma_0$.  Notice that now we are not allowed to assume that these two integrable systems share the same canonical coordinate system $(x,y,\lambda,\varphi)$  as we did in Propositions \ref{prop:5.2} and \ref{prop:5.3}.

Let $dF|_{\gamma_0} \ne 0$. We will say that $\Omega$ and $\wt\Omega$ {\em induce}
\begin{itemize}
\item {\em the same orientation} of $\gamma_0$ if the Hamiltonian flows of $F$ w.r.t. $\Omega$ and $\wt\Omega$ induce the same orientation of $\gamma_0$;
\item {\em the same coorientation} of $\gamma_0$ if  (the restrictions of) 
$\Omega$ and $\wt\Omega$ induce the same orientation of a (local) 2-dimensional surface in $\{F=F(\gamma_0)\}$ transversal to $\gamma_0$ (i.e., on a 2-dim Poincar\'e section).
\end{itemize}

Without loss of generality, we may (and will) assume that $\Omega$ and $\wt\Omega$ induce the same orientation and the same coorientation of $\gamma_0$. Indeed, we can easily achieve this condition  by using additional maps  $(x,y,\lambda,\varphi) \mapsto  (x,y,\lambda,-\varphi)$ and  $(x,y,\lambda,\varphi) \mapsto  (-x,y,\lambda,\varphi)$ (written in a canonical coordinate system from Proposition \ref{pro:coord}) that change respectively the orientation and coorientation without changing the functions $F$ and $H$.

As above we can define two natural action variables for each of these two integrable systems  $I(H,F)$, $I_\circ(H,F)$ and $\widetilde I(H,F)$, $\widetilde I_\circ(H,F)$. Here $I(H,F)$ and $\widetilde I(H,F)$ are smooth on a certain neighborhood $U(\gamma_0)$ and  are generators of  the Hamiltonian $S^1$-actions w.r.t. $\Omega$ and $\widetilde\Omega$ respectively.  

Alternatively, we may define $I(H,F)$  by 
$$
I(H,F) = \frac{1}{2\pi}\oint_{\gamma} \varkappa, \quad \mbox{where } d\varkappa = \Omega
$$ 
and $\gamma=\gamma_{H,F}$ is a closed cycle on the fiber $\mathcal L_{H,F}$ that is homotopic to $\gamma_0$  (recall that locally our fibration  can be understood as the direct product  of $S^1$ and a three-dimensional foliated domain $\mathcal V$ shown in Figure 2,   then  $\gamma_{H,F}$ can be taken of the form $S^1 \times \{P\}$ where $P\in \mathcal V$ is a point lying on the corresponding fiber).

The other action variable $I_\circ(H,F)$ is only defined on the family of ``narrow'' Liouville tori  corresponding to {\it small oriented loops}  $\mu_\circ = \mu_\circ (H, F)$ shown in Fig.~2:
$$
I_\circ (H,F)= \frac{1}{2\pi}\oint_{\mu_\circ}  \varkappa , \quad  \mbox{where }
d\varkappa = \Omega.
$$
In other words, $I_\circ(H,F)$ is a function defined on the ``swallow-tail'' domain on $\R^2(H,F)$ bounded by the bifurcation diagram $\Sigma$ (this definition coincides with \eqref{eq:action:narrow} up to, perhaps, changing the sign).

The actions $\widetilde I(H,F)$ and $\widetilde I_\circ(H,F)$ for the second system are defined in a similar way by integrating $\wt\varkappa$,  $d\wt\varkappa = \wt \Omega$,   over the same cycles $\gamma$ and $\mu_\circ$ with the same orientations.

\begin{thm} \label{thm:5.4}
Suppose that the singular fibration defined by the functions $H$ and $F$ is Lagrangian w.r.t. both the symplectic forms $\Omega$ and $\wt\Omega$. Suppose that $\Omega$ and $\wt\Omega$ induce the same orientation and the same coorientation of a parabolic orbit $\gamma_0$.
Then the following two statements are equivalent.

\begin{enumerate}

\item[(i)] In a tubular neighborhood of the parabolic orbit $\gamma_0$ there is a (real-analytic) diffeomorphism $\Phi$  such that

\begin{itemize}
\item $\Phi$ preserves $H$ and $F$;
\item $\Phi^* (\widetilde \Omega) = \Omega$.
\end{itemize}

\item[(ii)]  These two integrable systems have common action variables described above, i.e.,    
$$
I (H,F)= \widetilde I (H,F) + \const  \quad \mbox{and} \quad   
I_\circ(H,F)= \widetilde I_\circ (H,F). 
$$  
\end{enumerate}
\end{thm}

\begin{proof}

Suppose (ii) holds true. First of all we replace the functions $F$ and $H$ by new functions $\hat F$  and $\hat H$  satisfying the following conditions  (cf. Proposition  \ref{prop:2.1} and Remark \ref{rem:new1}):
\begin{itemize}

\item $\hat F = \pm I(H, F) + \mathrm{const}$  where $\pm$ and $\mathrm{const}$ are chosen in such a way that  $\hat F=0$ on the parabolic trajectory $\gamma_0$ and $\hat F<0$ on the swallow-tail domain of the bifurcation diagram;

\item $\hat H$ is chosen in such a way that the bifurcation diagram of  $\hat{\mathcal F}=(\hat F, \hat H)$ takes the standard form \eqref{eq:canonbd1}.    
\end{itemize}

After this we apply  Proposition \ref{pro:coord2} which says that formulas from Proposition \ref{pro:coord} holds true exactly for the functions $\hat F$ and $\hat H$.   In other words, we can introduce {\it two different} ``good''  coordinate systems $(x,y,\lambda, \varphi)$ and $(\tilde x, \tilde y, \tilde \lambda, \tilde \varphi)$  as in  Proposition \ref{pro:coord} for $(\hat H, \hat F, \Omega)$ and  $(\hat H, \hat F, \wt \Omega)$ respectively  (notice that  $\lambda = \tilde\lambda$ automatically  as both $\lambda$ and $\tilde \lambda$ coincide with $\hat F$).   

The next step is to consider the map $\Psi: (x,y,\lambda, \varphi) \mapsto (\tilde x, \tilde y, \tilde \lambda, \tilde \varphi)$ and after this continue working with the forms $\Omega$ and $\Psi^*(\widetilde\Omega)$.  Now $(x,y,\lambda, \varphi)$ is a common ``good'' coordinate system for the both systems and the conditions of Theorem \ref{thm:5.4} are still fulfilled for 
$\Omega$ and $\Psi^*(\widetilde\Omega)$.   
After this,  it remains to apply Proposition \ref{prop:5.3}  for the integrable systems $(\hat H, \hat F, \Omega)$ and $(\hat H, \hat F, \Psi^*(\widetilde\Omega))$.

The fact that (i) implies (ii) follows from the assumption that the symplectic forms $\Omega$ and $\wt \Omega$ induce the same  orientation and co-orientation on $\gamma_0$. Indeed this  implies that $\Phi$ preserves the homology class of $\gamma$ and $\mu_\circ$ on each ``narrow'' torus. Therefore  if we set $\varkappa = \Phi^*\wt\varkappa$ in the definition of the actions $I(H,F)$ and $I_\circ(H,F)$, then $I (H,F)= \widetilde I (H,F)$ and $I_\circ(H,F)= \widetilde I_\circ (H,F)$.
\end{proof}

In fact, we do not even need to mention $H$ and $F$ in the statement of Theorem \ref{thm:5.4} at all.  We may simply say:

\begin{thm} \label{thm:5.5}
Consider a singular fibration with a parabolic orbit $\gamma_0$ which is 
Lagrangian with respect to two symplectic structures $\Omega$ and $\widetilde\Omega$.
Suppose that $\Omega$ and $\wt\Omega$ induce the same orientation and the same coorientation of $\gamma_0$.
The necessary and sufficient condition
for the existence of a (real-analytic) diffeomorphism $\Phi$ in a 
tubular
neighborhood of $\gamma_0$ sending each fiber to itself
and such that $\Phi^* (\widetilde \Omega) = \Omega$ is that these two systems have common action variables in the sense that
for every closed cycle $\tau$ on any ``narrow'' torus   
we have  
$$
\oint_{\tau}  \varkappa =  \oint_{\tau}  \widetilde\varkappa, \ \ \mbox{for} \  
d\varkappa = \Omega,  \ d\widetilde\varkappa = \widetilde\Omega,
$$
where $\varkappa$ and $\widetilde\varkappa$ are chosen in such a way that $\oint_{\gamma_0}   \varkappa = \oint_{\gamma_0}  \widetilde\varkappa =0.
$\hfill $\square$
\end{thm}

Notice that due to analyticity it is sufficient to compare  the actions only on the family of ``narrow'' tori,  although $I$ and $\tilde I$ are defined on the whole neighborhood $U(\gamma_0)$.

\medskip

Finally, we want to relax the condition that each fiber goes to itself  (indeed, this assumption makes no sense at all if we want to compare parabolic orbits for two different integrable systems).

Assume that we are given two integrable systems with parabolic orbits $\gamma_0$ and $\wt\gamma_0$, respectively. For the both systems we consider the bifurcation diagrams (or bifurcation complexes), $\Sigma$ and $\wt\Sigma$ respectively, and the 
``swallow-tail domains'' corresponding to the families of ``narrow'' Liouville tori.  On each of these domains
we have two actions $I$ and $I_\circ$  (as functions of $H$ and $F$) and correspondingly $\wt I$ and $\wt{I}_\circ$  (as functions of $\wt H$ and $\wt F$) defined as above.  Without loss of generality we will assume that these action variables are ``normalised'' in such a way that 
\begin{itemize}
\item all of them vanish at  the corresponding cusp point, 
\item $I_\circ$ and $\wt{I}_\circ$ are positive on the corresponding ``swallow-tail'' domains,
\item $I$ and $\tilde I$ are negative on the corresponding ``swallow-tail'' domains.
\end{itemize}
Combining Theorem \ref {thm:5.4} with Proposition \ref {prop:lift}  we obtain

\begin{thm} \label{thm:5.6}
The necessary and sufficient condition for the existence of a real-analytic fiberwise symplectomorphism $\Phi: U(\gamma_0) \to \widetilde U (\widetilde{\gamma}_0)$ between some tubular neighborhoods $U(\gamma_0),\widetilde U (\widetilde{\gamma}_0)$ of the parabolic orbits $\gamma_0,\wt\gamma_0$ is that these two systems have common action variables in the sense that 
there is a real-analytic diffeomorphism
\begin{equation}
\label{eq:mapHF}
\phi: (H, F) \mapsto (\widetilde H, \widetilde F) 
\end{equation}
between some neighborhoods of the cusp points $(H(\gamma_0),F(\gamma_0))$ and $(\wt H(\wt\gamma_0),\wt F(\wt\gamma_0))$ in $\R^2$
that 
\begin{itemize}
\item respects the bifurcation diagrams together with their partitions into hyperbolic and elliptic branch\footnote{Equivalently, we may say that $\phi$  defines a (local) homeomorphism between the corresponding bifurcation complexes.}:
$$
\phi(\Sigma)=\wt\Sigma, \quad \mbox{moreover} \quad
\phi(\Sigma_{\mathrm{ell}})=\wt\Sigma_{\mathrm{ell}} \quad \mbox{and} \quad
\phi(\Sigma_{\mathrm{hyp}})=\wt\Sigma_{\mathrm{hyp}},
$$
\item and preserves the action variables described above: $I=\wt I\circ\phi$ and $I_\circ=\wt I_\circ\circ\phi$, i.e., for the action variables defined on the ``swallow-tail domains'' we have 
$$
I(H,F)=\widetilde I (\widetilde H(H,F), \widetilde F(H,F)) \quad \mbox{and} \quad I_\circ(H,F)=\widetilde I_\circ (\widetilde H(H,F), \widetilde F(H,F)).
\eqno \square
$$
\end{itemize}
\end{thm}

The latter conclusion basically means that the only symplectic invariants of hyperbolic orbits are {\it action variables}.  This conclusion does not provide any tools to decide whether a suitable map \eqref{eq:mapHF} (making the actions equal) exists or not, but some necessary conditions can be easily found.   Some of them have
been already described in Section \ref{sec:4},  e.g., the function $f(\cdot)$ from Proposition \ref{prop:3.3.11}.  This function is a symplectic invariant of a parabolic singularity which ``corresponds'' to the level $\lambda =0$, where
$\lambda$, as above, denotes the first action variable $I(H,F)$.  We now want to describe another non-trivial symplectic invariant which will be a function $h(\lambda)$, $\lambda <0$.

Since $\lambda=\lambda(H,F)$ is a real-analytic function, we can consider it as a parameter on the hyperbolic branch $\Sigma_{\mathrm{hyp}}$ of the bifurcation diagram $\Sigma$.   Consider  $I_\circ (H,\lambda)$ as a function of $H$ (with $\lambda$ as a parameter). This function is defined on the interval   
$$
\left(-2(-\lambda)^{3/2}/(3\sqrt3),2(-\lambda)^{3/2}/(3\sqrt3)\right),
$$
is strictly increasing from 0 to its maximum attained on the hyberbolic branch. We denote it by  $h(\lambda)=\max_H I_\circ (H,\lambda)$.  Obviously, $h(\lambda)$  does not depend on the choice of commuting functions $H$ and $F$ defining the Lagrangian fibration, so that  $h(\lambda)$ can be considered as a symplectic invariant of a parabolic singularity. 

The problem of an explicit description of a complete set of symplectic invariants
is equivalent,  as shown above, to the analysis of the asymptotics of the function $I_\circ (H,\lambda)$. More precisely,  we should describe invariants of such functions under (real-analytic) transformations of the form $(H,\lambda)  \mapsto \bigl( \widetilde H (H,\lambda), \widetilde \lambda = \lambda\bigr)$.


\section{Semi-local  invariants of cusp singularities} \label{sec:6}

Finally, we want to describe semi-local invariants of cusp singularities.  In other words, we now consider a saturated neighborhood of a compact singular fiber $\mathcal L_0$ containing a parabolic orbit, i.e., cuspidal torus.  We assume that this fiber  contains no other critical points, so that the topology of the fibration in a neighborhood of $\mathcal L_0$ is  standard and illustrated in Figure 3.   This Figure also shows the bifurcation complex, i.e., the base of this fibration, which consists of two  2-dimensional strata (attached to each other along $\Sigma_{\mathrm{hyp}}$,  one of the branches of the bifurcation diagram $\Sigma$ that corresponds to the family of hyperbolic orbits).  Each stratum represents a family of Liouville tori and therefore we can naturally assign a pair of action variables to each of them.  Our goal is to show that fibrations with the same actions are symplectomorphic.  

 In a neighborhood  $U(\mathcal L_0)$ of the singular fiber $\mathcal L_0$,  on all neighboring Liouville tori we can choose a natural basis of  cycles in the first homology group of $H_1(T^2_{F,H}, \mathbb Z)$ 
where $T^2_{H,F}$ is the Liouville torus defined by fixing the values of the integrals $F$ and $H$ respectively.   These cycles are shown in Figure 3.  One of them corresponds to the $S^1$-action defined on $U(\mathcal L_0)$  (in Figure 3, this cycle $\gamma$ is denoted by $S^1$).  The other cycle can be obtained by considering a global 3-dimensional cross-section to this $S^1$-action.  Since this $S^1$-action  (and the corresponding $S^1$-fibration) is topologically trivial, such a cross section exists.  It is illustrated on the left in Figure 3 and denoted by $V$  so that we may think of  $U(\mathcal L_0)$ as the direct product $V \times S^1 = U(\mathcal L_0)$.    Each Liouville torus  $T^2_{F,H}$ intersects $V$ along a closed curve  (these curves are shown in Figure 3) and this curve is taken as the second basis cycle $\mu$ in $H_1(T^2_{F,H}, \mathbb Z)$.  More precisely, we need to take into account that for a point $(F,H)$ from the swallow-tail zone,  we will have two disjoint Liouville tori.   The corresponding cycles will be denoted by $\mu$ and $\mu_\circ$, where $\mu_\circ$ is used for the vanishing cycle on the family of ``narrow'' tori, the other, i.e., $\mu$,  corresponds to a ``wide'' torus. 

Notice that the first cycle $\gamma$ is uniquely defined by the $S^1$-action. The cycle $\mu_\circ$ is also well defined by the topology of the fibration  (as a vanishing cycle). The other cycle $\mu$ is not.  It is easy to see that $\mu$ is defined up to the transformation of the form $\mu \mapsto \mu + k\gamma$,   $k\in \mathbb Z$.   This is caused by ambiguity in the choice of the cross-section $V$ which can be chosen in many homotopically different ways (this phenomenon is discussed  and explained in details in \cite{BolsFomBook}).  

Summarizing,  on each stratum of the bifurcation complex, we have a pair of action variables  $I_\gamma, I_\mu$  and $I_\gamma, I_{\mu_\circ}$ (the latter for the swallow-tail stratum).   Each of these functions can be treated as a real-analytic function of $H$ and $F$.  In fact,  we have already considered the actions $I_\gamma$ and $I_{\mu_\circ}$ in the previous Section \ref{sec:5}, where they were denoted by $I(H,F)$ and $I_\circ (H,F)$.   We will keep this notation here, i.e., we set  $I_\gamma = I$, $I_{\mu_\circ} = I_\circ$. The remaining action will be denoted by $I_\mu$ so that we have 3 action variables $I, I_\circ$ and $I_\mu$.  The first two of them are well-defined,  but $I_\mu$ is defined modulo transformation $I_\mu \mapsto I_\mu + k I$. 

Also notice that $I(H,F)$ is real-analytic everywhere  (strictly speaking we need to distinguish this action for the families of ``narrow'' and ``wide'' tori, but due to real-analyticity  $I(H,F)$, as function of $H$ and $F$,  is the same for the both families).  The function  $I_\mu (H, F)$  is defined and is real-analytic everywhere except for the hyperbolic branch $\Sigma_{\mathrm{hyp}}$ of the bifurcation diagram.  When approaching  $\Sigma_{\mathrm{hyp}}$ the function tends to certain finite limits, but these limits from above and from below are different.   The function $I_\circ$ is defined on the swallow-tail domain and is continuous on  its closure.

\vskip10pt


\begin{figure}[htbp]\label{fig3}
\begin{center}
\includegraphics[width=0.80\textwidth]{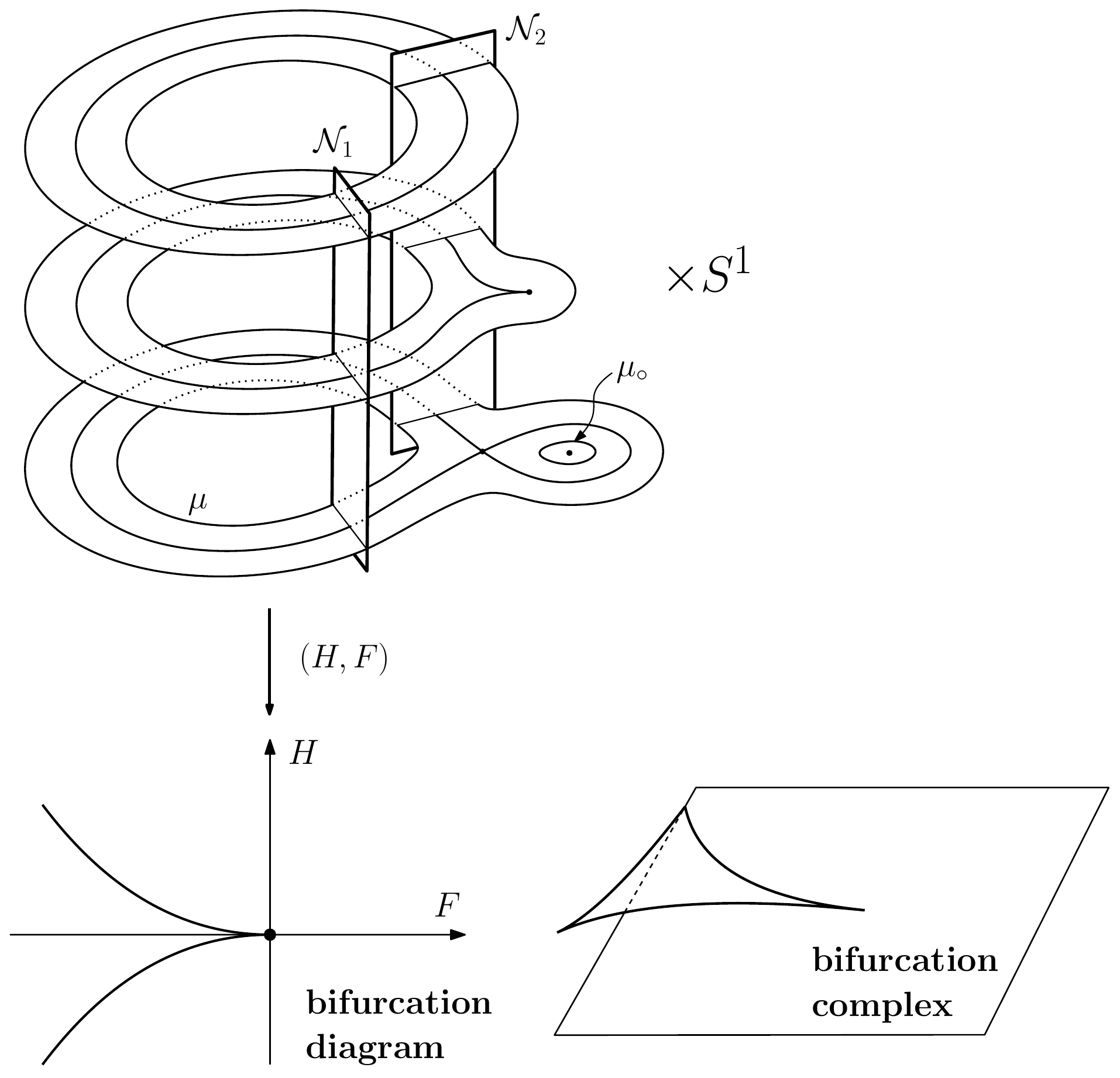}
\end{center}
\caption {Singular fibration near a cuspidal torus} \label {fig:3}
\end{figure}

Our final result basically states that the systems with equal actions are symplectomorphic.  We will give two versions of this result.
Consider two integrable Hamiltonian systems $\bigl(H,F,\Omega, U(\mathcal L_0)\bigr)$ and $\bigl(\widetilde H,\widetilde F,\widetilde \Omega, \widetilde U(\widetilde{\mathcal L}_0)\bigr)$ defined on some neighborhoods\footnote{We do not specify the sizes of these neighborhoods, but assume that they are sufficiently small. In other words, we are talking about germs of fibrations and germs of maps.} of cuspidal tori $\mathcal L_0$ and $\widetilde{\mathcal L}_0$.

\begin{thm}
\label{th:6.1}
Assume that there is a fiberwise diffeomorphism $\Psi : U(\mathcal L_0) \to \widetilde U(\widetilde{\mathcal L}_0)$ that preserves the actions in the sense that for every cycle $\tau \subset \mathcal L_{H,F}$ we have
$$
\oint_{\Psi(\tau)} \wt \varkappa = \oint_{\tau} \varkappa
$$ 
for some 1-forms $\varkappa$ and $\wt \varkappa$ satisfying $d\varkappa = \Omega$, $d\wt\varkappa=\wt\Omega$.
Then there exists a fiberwise symplectomorphism  $\Phi: U(\mathcal L_0)\to \widetilde U(\widetilde{\mathcal L}_0)$. \end{thm}

\begin{rem}\label{rem:converse}
The converse statement is obviously true, since a fiberwise symplectomorphism $\Phi$ preserves the actions:
$$
\oint_{\Phi(\tau)} \wt \varkappa = \oint_{\tau} \varkappa
$$ 
where $\varkappa$ and $\wt \varkappa$ are related by $\Phi^*\wt \varkappa = \varkappa$.
\end{rem}

A stronger version is as follows.  
For each system we compute the actions $I(H,F)$, $I_\circ(H,F)$ and $I_\mu (H,F)$ and respectively  $\widetilde I(\widetilde H,\widetilde F)$, $\widetilde I_\circ(\widetilde H,\widetilde F)$ and $\widetilde I_{\widetilde \mu} (\widetilde H,\widetilde F)$  for the second system as explained above.

\begin{thm}
\label{th:6.2}
 Assume that there is a local real-analytic diffeomorphism  $\phi : (H,F) \mapsto (\wt H, \wt F)$,  $\wt H=\wt H(H,F)$ and $\wt F=\wt F(H,F)$
that 
\begin{itemize}
\item respects the bifurcation diagrams together with their partitions into hyperbolic and elliptic branches:
$$
\phi(\Sigma)=\wt\Sigma, \quad \mbox{moreover} \quad
\phi(\Sigma_{\mathrm{ell}})=\wt\Sigma_{\mathrm{ell}} \quad \mbox{and} \quad
\phi(\Sigma_{\mathrm{hyp}})=\wt\Sigma_{\mathrm{hyp}},
$$
\item 
 and makes the actions equal (for some choice of $\mu$ and $\widetilde\mu$):
$$
\begin{aligned}
I (H,F) &= \wt I (\wt H(H,F), \wt F(H,F)), \\
I_\circ (H,F) &= \wt I_\circ (\wt H(H,F), \wt F(H,F)) \ \mbox{and} \\
I_\mu (H,F) &= \wt I_{\widetilde \mu}  (\wt H(H,F), \wt F(H,F)).
\end{aligned}
$$
\end{itemize}
Then there exists a fiberwise symplectomorphism  $\Phi: U(\mathcal L_0)\to \widetilde U(\widetilde{\mathcal L}_0)$.
\end{thm}

\begin{rem}\label{rem:converse2}
Notice that the converse statement is also true: a fiberwise symplectomorphism $\Phi:U(\mathcal L_0)\to \widetilde U(\widetilde{\mathcal L}_0)$ induces a diffeomorphism $\phi$ between the bases of the fibrations which automatically satisfies the properties above (where the choice of $\wt \mu$ is induced by $\Phi$ and $\mu$).
\end{rem}

\begin{rem}\label{rem:3}
{\rm We can rewrite this statement in a slightly different and shorter way.     For each of the above integrable systems we may consider the momentum map $\pi :  U(\mathcal L_0) \to B \subset \R^2(H, F)$ and $\pi :  \widetilde U(\widetilde{\mathcal L}_0) \to \widetilde B \subset \R^2(\widetilde H, \widetilde F)$,  where $B$ and $\widetilde B$ are some neighborhoods of the corresponding cusp points of the bifurcation diagrams.  Then we can think of the actions as functions on $B$ (more precisely on the corresponding domains defined by the bifurcation diagrams).   Then Theorem~\ref{th:6.2} can be rephrased as follows:

{\it Assume that there exists a local real-analytic diffeomorphism $\phi : B \to \widetilde B$ respecting the bifurcation diagrams $\Sigma$ and $\wt \Sigma$ and such that 
$I = \widetilde I\circ \phi$,  $I_\circ = \widetilde {I}_\circ \circ \phi$ and $I_\mu = \widetilde I_{\widetilde\mu} \circ \phi$. Then there exist a fiberwise symplectomorphism  $\Phi: U(\mathcal L_0)\to \widetilde U(\widetilde{\mathcal L}_0)$.}
}\end{rem}

The proof of this theorem is based on the following lemma. Consider two (non-singular) integrable systems $(H, F, \Omega)$ and $(\widetilde H, \widetilde F, \widetilde \Omega)$ defined in some neighborhoods  $T^2 \times B$ and  $\widetilde T^2\times \widetilde B$ of regular Liouville tori. Here $B$ and $\tilde B$ are 2-dimensional discs viewed as the bases of the corresponding (regular) Lagrangian fibrations  endowed with induced integer affine structures   (action variables).  The functions $(H,F)$ and $(\widetilde H, \widetilde F)$ are treated as smooth functions  on $B$ and $\widetilde B$ respectively. We also consider the Hamiltonian $\R^2$-actions  $\sigma^{(t_1, t_2)}$ and $\widetilde{\sigma}^{(t_1, t_2)}$, $(t_1, t_2)\in \R^2$ generated by the commuting functions $(H, F)$ and $(\tilde H, \tilde F)$.  Here $\sigma^{(t_1, t_2)}$ denotes the composition of the Hamiltonian shifts along vector fields $X_H$ and $X_F$ by time $t_1$ and time $t_2$ respectively.  Similarly for $\tilde\sigma^{(t_1, t_2)}$.

\begin{lemma}\label{lem:6.3}
Assume that we have a real-analytic diffeomorphism  
$$
\phi : B \to \tilde B, \quad  
$$  
which provides an (integer) affine equivalence between $B$ and $\widetilde B$.   Let   $H = \wt H \circ\phi$ and  $F = \wt F \circ\phi$   and consider two Liouville tori $T_p =T^2 \times \{p\}$   and  $\wt T_{\phi(p)} = \widetilde T^2 \times \{\phi (p)\}$ where $p\in B$ is an arbitrary point {\rm(}in other words, these tori correspond to each other under the map $\phi: B \to\widetilde B${\rm)}.    Let $x\in T_p$ and $\tilde x \in \widetilde T_{\phi(p)}$ be arbitrary two points from these fibers.   

Then $\sigma^{(t_1,t_2)} (x) = x$   {\rm(}or more generally $\sigma^{(t_1, t_2)}(x) = \sigma^{(t'_1, t'_2)}(x)${\rm)}  if and only if 
$\widetilde \sigma^{(t_1,t_2)} (\tilde x) = \tilde x$  {\rm(}respectively $\widetilde\sigma^{(t_1, t_2)}(\tilde x) = \widetilde\sigma^{(t'_1, t'_2)}(\tilde x)${\rm)}. 
\end{lemma}

\begin{proof}  We will give a proof of this statement in the case of $n$ degrees of freedom.   Recall that $B$ and $\widetilde B$ are endowed with integer affine structures induced by the action variables.  By definition,  $\phi: B \to \widetilde B$ is an  (integer)  affine equivalence if $\phi$ sends ``actions to actions''. More precisely,
let $\wt I_1, \dots, \wt I_n$  be action variables for $\tilde B$,  which means that these functions define the Hamiltonian action of the {\it standard} torus $\R^n/\Gamma_0$ where $\Gamma_0=\mathbb Z^n$ is the standard  integer lattice in $\R^n$. 
We say that $\phi :B \to \wt B$ is an affine equivalence, if $I_1=\wt I_1 \circ \phi, \dots , I_n=\wt I_n \circ \phi$ are action variables on $B$.

If in Lemma  \ref{lem:6.3} instead of $(H, F)$ and $(\widetilde H, \widetilde F)$ we consider $(I_1, I_2)$ and $(\wt I_1, \wt I_2)$, then the statement  is obvious:   both relations $\sigma^{(t_1, t_2)} (x) = x$ and $\wt \sigma^{(t_1, t_2)} (\wt x) = \wt x$  simply mean that $(t_1, t_2)$ belongs to the standard integer lattice, i.e., $t_1, t_2 \in \mathbb Z$.

Let us see what happens if take arbitrary functions  $(H, F)$  or, more generally, $(F_1, F_2, \dots F_n)$ in the case of $n$ degrees of freedom.  The relation $\sigma^{(t_1, \dots, t_n)} x = x$  means that  $(t_1,\dots, t_n)$ belongs to the period lattice $\Gamma \subset \R^n$ which is the stationary subgroup of $x$ in the sense of the Hamiltonian $\R^n$-action generated by $F_1, F_2, \dots F_n$.   Since this lattice is the same for any point $x$ from a fixed torus $T_p$, $p \in B$,  we may denote it by $\Gamma(T_p)$.   This lattice is not standard anymore and it depends on two things,  the torus  $T_p$  (or just a point $p\in B$) and the generators $F_1, F_2, \dots F_n$ of the Hamiltonian $\R^n$-action.

If we know the expressions of $F_1, \dots, F_n$ in terms of the actions $I_1,\dots, I_n$, then the lattice $\Gamma(T_p)$ is easy to describe. Namely: 
$$
\Gamma(T_p) =  \Gamma_0 \cdot J^{-1}(p),
$$
where $\Gamma_0$ is the standard integer lattice and $J(p)$ denotes the Jacobi matrix  $J(p) = \left(  J^i_j= \frac{\partial F_i}{\partial I_j}|_p\right)$.  In more details,
$$
(t_1, \dots , t_n) \in \Gamma(T_p) \quad \mbox{if and only if} \quad (t_1, \dots , t_n) = (k_1,\dots, k_n) \cdot J^{-1}(p)
$$
for $ (k_1,\dots, k_n)\in \Gamma_0$,  i.e., for some vector with integer components $k_i  \in \mathbb Z$. 

The same, of course, holds for $\wt x \in \wt T_{\phi(p)}$,  that is 
$$
(t_1, \dots , t_n) \in \Gamma(\wt T_{\phi(p)}) \quad \mbox{if and only if} \quad (t_1, \dots , t_n) = (k_1,\dots, k_n) \cdot \wt J^{-1}(\phi(p))
$$
where $\wt J(\phi(p)) = \left( \wt J^i_j= \frac{\partial \wt F_i}{\partial \wt I_j}|_{\phi(p)}\right)$.  It remains to notice that under our assumptions these matrices coincide.   The reason is obvious:  since $I_k = \wt I_k \circ \phi$ and also $F_i = \wt F_i \circ \phi$,  we see that 
$F_i = f_i(I_1,\dots,I_n)$ implies that $\wt F = f_i (\wt I_1, \dots , \wt I_n)$, i.e.,   $F_i$ depends on $I_1,\dots,I_n$  exactly in the same way as 
$\wt F_i$ depends on $\wt I_1,\dots, \wt I_n$ so that the corresponding partial derivatives (being computed at $p$ and $\phi(p)$, i.e., at those points for which $(I_1,\dots,I_n) = (\wt I_1,\dots, \wt I_n)$) obviously coincide. In other words, we have proved that $\Gamma(T_p) = \Gamma (\wt T_{\phi(p)})$, which is equivalent to our statement. \end{proof}

This lemma implies the following two extension results.

Under the assumptions and notation from Lemma \ref{lem:6.3}, assume that $N$ and $\widetilde N$ are Lagrangian (real-analytic) sections of the Lagrangian fibrations  $\pi: T^2 \times B \to B$ and $\widetilde \pi: \widetilde T^2 \times \widetilde B \to \widetilde B$ respectively.   Since the sections $N$ and $\widetilde N$ can be naturally identified with the bases $B$ and $\widetilde B$,  the map $\phi : B \to \widetilde B$  induces a natural map between $N$ and $\tilde N$ which we denote by the same letter $\phi : N \to \widetilde N$.  For any point $y\in T^2 \times B$ we can find  (not uniquely!)  $(t_1(y), t_2(y))\in \R^2$ such that $x=\sigma^{(t_1(y), t_2(y))}(y) \in N$.  Consider the map $\Phi: T^2 \times B \to \widetilde T^2\times \widetilde B$ defined by
$$
\Phi (y) = \wt \sigma^{(-t_1(y), -t_2(y))}(\phi(x)), \quad \mbox{where }  x=\sigma^{(t_1(y), t_2(y))}(y) \in N.
$$

\begin{cor} 
\label{cor:6.4}
The map $\Phi (y)$ is well defined  and is a fiber-wise real-analytic diffeomorphism  satisfying $\Phi^*(\widetilde\Omega) = \Omega$.
\end{cor}

\begin{proof} The fact that $\Phi$ is well defined (i.e., does not depend on the choice of $(t_1, t_2)\in \R^2$  with the property $\sigma^{(t_1, t_2)}(y) \in N$) follows from Lemma \ref{lem:6.3}.  To show that $\Phi$ is symplectomorphism, i.e., $\Phi^*(\widetilde\Omega) = \Omega$,  we notice that the position of each point $y \in T^2 \times B$ is defined by the values of $H, F$   (which can be understood as coordinates on $B$) and $t_1, t_2$ (which can be understood as coordinates on the torus $T^2$ with the ``origin'' $(0,0)$ located on $N$).  These four functions define a canonical coordinate system, i.e.,
$$
\Omega = dH\wedge dt_1 + dF\wedge dt_2. 
$$
A similar canonical coordinate system  $\tilde H, \tilde F, \tilde t_1, \tilde t_2$ can be defined on $\wt T^2 \times \wt B$ by using the action $\wt\sigma$ and the Lagrangian section $\wt N$.  It remains to notice that our map $\Phi$ in these coordinate systems, by construction, takes the form
$\wt H = H$, $\wt F = F$, $\wt t_1=t_1$, $\wt t_2 = t_2$.  \end{proof}

Let $U\subset T^2 \times B$ be an open subset such that the intersection of $U$ with each fiber is connected and non-empty.  Let  $\Phi_{\mathrm{loc}} :  U \to \widetilde U$ be a real-analytic fiber-wise diffeomorphism with a certain open subset $\widetilde U \subset \widetilde T^2 \times \widetilde B$ such that $\Phi_{\mathrm{loc}}^*(\widetilde\Omega) = \Omega$.  Since $\Phi_{\mathrm{loc}}$ is fiberwise and $U$ intersects each fiber,   $\Phi_{\mathrm{loc}}$ induces a real-analytic map $\phi$ between the bases $B$ and $\widetilde B$.  

\begin{cor}\label{cor:6.5}
$\Phi_{\mathrm{loc}}$ can be extended  up to 
a real-analytic fiber-wise diffeomorphism  $\Phi: T^2 \times B \to \widetilde T^2 \times \widetilde B$ with the property $\Phi^*(\widetilde\Omega) = \Omega$   if and only $\phi: B \to \widetilde  B$ is an integer affine equivalence.  
\end{cor}

\begin{proof}   First of all we notice that such an extension (if it exists) is always unique.  Indeed,   since $\Phi$ is a symplectomorphism,  we have
$$
\Phi \circ \sigma^{(t_1, t_2)} =  \wt \sigma^{(t_1, t_2)} \circ \Phi,
$$
where $\sigma$ and $\wt \sigma$ are Hamiltonian $\R^2$-actions generated by $H,F$  and  $\wt H=H\circ \Phi^{-1}$, $\wt F= F\circ \Phi^{-1}$ respectively.   Therefore for any $y \in T^2 \times B$,   its image $\Phi (y)$ is uniquely defined by:
\begin{equation}
\label{eq:25}
\Phi(y) = \wt \sigma^{(t_1, t_2)} \circ \Phi_{\mathrm{loc}} \circ \sigma^{(-t_1, -t_2)} (y), 
\end{equation}
where $(t_1, t_2)$ are chosen in such a way that $\sigma^{(-t_1, -t_2)} (y)\in U$  (such $(t_1, t_2)\in\R^2$ exists as each orbit of the action $\sigma$ has a non-trivial intersection with $U$).   Moreover,  this formula can be understood as an explicit formula for the required extension.  
In a neighborhood of every point $y$,  the expression $\wt \sigma^{(t_1, t_2)} \circ \Phi_{\mathrm{loc}} \circ \sigma^{(-t_1, -t_2)}$  (with fixed $(t_1, t_2)$) is a composition of three real-analytic fiberwise symplectomorphisms.  So the only condition we need to check is that formula \eqref{eq:25} is well defined, i.e.,  does not depend on the choice of $(t_1, t_2)\in\R^2$. 

Assume that
$$
y = \sigma^{(t_1, t_2)} (x) = \sigma^{(t'_1, t'_2)} (x') \quad \mbox{with } x,x'\in U.  
$$

We need to check that
\begin{equation}
\label{eq:26}
\wt\sigma^{(t_1,t_2)} \circ \Phi_{\mathrm{loc}} \circ \sigma^{(-t_1,-t_2)} (y) = 
\wt\sigma^{(t'_1,t'_2)} \circ \Phi_{\mathrm{loc}} \circ \sigma^{(-t'_1,-t'_2)} (y)
\end{equation}
or, equivalently,
\begin{equation}
\label{eq:27}
\wt\sigma^{(t_1,t_2)} \circ \Phi_{\mathrm{loc}} (x) = 
\wt\sigma^{(t'_1,t'_2)} \circ \Phi_{\mathrm{loc}} (x').
\end{equation}
By our assumption,  the intersection of $U$ with each torus (interpreted now as an orbit of $\sigma$) is connected,   therefore there exists a continuous curve $(\varepsilon_1(s), \varepsilon_2(s))$,  $s\in [0,1]$ and  $\varepsilon_1(0)= \varepsilon_2(0)=0$ such that
$$
\sigma^{(\varepsilon_1(s), \varepsilon_2(s))} (x) \in  U \quad \mbox{for all } s\in[0,1] \quad \mbox{and} \quad \sigma^{(\varepsilon_1(1), \varepsilon_2(1))} (x)= x'.
$$
Since $\Phi_{\mathrm{loc}}$ is a fiberwise symplectomorphism, we have 
$$
\Phi_{\mathrm{loc}} \circ \sigma^{(\varepsilon_1(s), \varepsilon_2(s))} (x) =  \wt \sigma^{(\varepsilon_1(s), \varepsilon_2(s))} \circ \Phi_{\mathrm{loc}} (x)
$$
for any $s$ and, in particular,
$$
\Phi_{\mathrm{loc}} (x') =  \wt \sigma^{(\varepsilon_1(1), \varepsilon_2(1))} \circ \Phi_{\mathrm{loc}} (x).
$$

Hence \eqref{eq:27} can be rewritten as 
\begin{equation}
\label{eq:28}
\wt\sigma^{(t_1,t_2)} \bigl( \Phi_{\mathrm{loc}} (x)\bigr)=\wt\sigma^{(t'_1+\varepsilon_1(1),t'_2+\varepsilon_2(1))} \bigl( \Phi_{\mathrm{loc}} (x)\bigr).
\end{equation}

On the other hand,  since $ \sigma^{(t_1, t_2)} (x) = \sigma^{(t'_1, t'_2)} (x')$, we also have 
\begin{equation}
\label{eq:29}
\sigma^{(t_1, t_2)} (x) = \sigma^{(t'_1+\varepsilon_1(1), t'_2+\varepsilon_2(1))} (x).
\end{equation}

According to Lemma \ref{lem:6.3},   if $\phi$ is an affine equivalence then \eqref{eq:29} implies  \eqref{eq:28} and therefore \eqref{eq:26}, as needed.
 
 The necessity of  the condition that  $\phi: B \to \widetilde  B$ is an integer affine equivalence is obvious:   every fiberwise symplectomorphism induces an affine equivalence between $B$ and $\wt B$.
\end{proof}

We now use Corollary \ref{cor:6.5}  to prove Theorem \ref{th:6.2}. 

\begin{proof} First we apply Theorem \ref{thm:5.6} which guarantees the existence of a real-analytic fiberwise diffeomorphism $\Phi_{\mathrm{loc}}$ \footnote{In Theorem \ref{thm:5.6}, this map was denoted by $\Phi$.} between some neighborhoods of parabolic trajectories $\gamma_0 \subset \mathcal L_0$ and $\wt\gamma_0\subset \wt{\mathcal L}_0$.  We now need to extend $\Phi_{\mathrm{loc}}$ up to the desired fiberwise 
symplectomorphism $\Phi: U(\mathcal L_0)\to \widetilde U(\widetilde{\mathcal L}_0)$.  

According to Corollary \ref{cor:6.5} such an extension exists for all Liouville tori  (more precisely,  we only need to consider ``wide'' Liouville tori  because all ``narrow'' Liouville tori are already contained in the domain of $\Phi_{\mathrm{loc}}$).  Thus, it remains to explain why this map can be extended by continuity to each singular fiber.   

 On  Fig.~3 we can see the domain $U$ on which $\Phi_{\mathrm{loc}}$ is already defined and the complementary domain $W$ to which $\Phi_{\mathrm{loc}}$ should be extended.  Without loss of generality we may assume that  both domains are bounded by the sections $\mathcal N_1$ and $\mathcal N_2$.  Namely, $U$ is located to the right of $\mathcal N_1$ and $\mathcal N_2$ and contains all singular orbits including the parabolic one.  The complimentary domain $W$ is located to the left of $\mathcal N_1$ and $\mathcal N_2$ and contains no singularities at all.     
 
 Let $y \in W$ be an arbitrary point located on one of singular fibers and $V(y)$ be a sufficiently small neighborhood of $y$.   Then there exists $(t_1, t_2)\in \R^2$ such that $\sigma^{(-t_1, -t_2)}(V(y)) \subset U$ and we may apply our extension formula \eqref{eq:26} to define $\Phi$ on $V(y)$.    Obviously, this formula defines a real-analytic fiberwise (local) symplectomorphism from $V(y)$ to its image in $\widetilde U(\widetilde{\mathcal L}_0)$ and moreover, due to the uniqueness of such an extension, this map coincides with $\Phi$ that has been already defined on non-singular fibers (Liouville tori).   This is equivalent to saying that $\Phi$ can be naturally extended (by continuity) from Liouville tori to  all singular fibers. This completes the proof.  \end{proof}

\begin{rem}  Our final remark is that the statement of Theorem \ref{th:6.2} given in Remark  \ref{rem:3} can be also understood in terms of natural affine structures defined on $B$ and $\widetilde B$.   

A necessary and sufficient condition for the existence of a {\it semi-local} fiberwise symplectomorphism  $\Phi: U(\mathcal L_0) \to \wt U(\wt{\mathcal L}_0)$ between neighborhoods of two cuspidal tori $\mathcal L_0$ and $\wt{\mathcal L}_0$ is that the corresponding bases $B$ and $\widetilde B$ are {\it locally} equivalent as manifolds with singular integer affine structures.  Moreover, every affine equivalence $\phi:B \to \wt B$ can be lifted up to a fiberwise symplectomorphism $\Phi$. 

Thus, our paper gives a partial answer to Problem 27 from the collection \cite{BITOpenProb} of  open problems in the theory of finite-dimensional integrable systems.  
\end{rem}


\section{Appendix}


In this appendix we give a formal proof of the statement made in Remark \ref{rem:new1},  namely we prove the following

\begin{proposition}
\label{prop:ap1}
Let $P$ be a parabolic point of a momentum mapping $\mathcal F=(F,H): M^4 \to \R^2$ in the sense of Definition \ref{def:parabolic}  we   {\rm(}in particular,  $dF(P)\ne 0${\rm)} and 
\begin{equation}
\label{eq:transformation}
\wt H = \wt H (H,F), \quad \wt F = \wt F (H,F)
\end{equation} 
be a non-degenerate transformation such that $d\wt F(P) \ne 0$.   Then $P$ is still parabolic w.r.t. $\wt H$ and $\wt F$.
\end{proposition}

\begin{proof}

First of all, we notice that in Definition \ref{def:parabolic}  we can replace $H$ by $H-\mathrm{const} F$ and, in particular,   by $H - kF$  where $k\in \R$ is chosen in such a way that $d(H-kF)=0$.  In other words, without loss of generality we may assume that $dH(P)=0$   and similarly for $\wt H$.  
Under this additional assumption, the quadratic differential $d^2 H(P)$ makes sense on the whole tangent space $T_PM^4$. Taking into account that the tangent space to the hypersurface $\{ F = F(P)\}$ coincides with the kernel of the differential $d\mathcal F$  (here we use the fact that $\ker d\mathcal F (P) = \ker dF(P) = T_P\{ F = F(P)\}$), we can reformute the first condition (i) as follows:

(i)  the rank of the restriction $d^2 H (P)|_{\ker d\mathcal F}$ equals 1.

The advantage of such a reformulation is that now this condition does not depend on the choice of $F$ at all.   Therefore to verify the invariance of Condition (i) w.r.t. transformation \eqref{eq:transformation}, it is sufficient to prove 

\begin{lemma}
\label{lem:ap1}
Let $\wt H = \wt H(H, F)$  and  $d\wt H (P) = 0$,  then the forms $d^2 H (P)|_{\ker d\mathcal F}$ and $d^2 \wt H (P)|_{\ker d\mathcal F}$ are proportional with a non-zero factor.
\end{lemma}

\begin{proof}
It is sufficient to compare the Taylor expansions of  $H$ and $\tilde H$  at the point $P$ up to second order terms.   Let 
$$
\Delta \wt H \simeq  a_1 \Delta H + a_2 \Delta F + a_{11} \Delta H^2 + 2 a_{12} \Delta H \Delta F  +  a_{22} \Delta F^2 \dots 
$$
Since $d\wt H(P) = dH(P) = 0$, we conclude that $a_2 = 0$ and $a_1\ne 0$ and obtain:
\begin{equation}
\label{eq:TaylorH}
\Delta \wt H \simeq \frac{a_1}{2} d^2 H (\Delta x, \Delta x) + a_{22} \left( dF(\Delta x)\right) ^2 + \dots
\end{equation}
(all the other terms are of order $\ge 3$ and we omit them) or equivalently:
\begin{equation}
\label{eq:seconddiff}
d^2 \wt H = a_1 d^2 H  + 2a_{22}  dF \otimes dF
\end{equation}
(the differentials and second differentials are taken at the point $P$). Restricting to $\ker dF = \ker d\mathcal F$,  we get $d^2 \wt H (P)|_{\ker d\mathcal F} = a_1 d^2 \wt H (P)|_{\ker d\mathcal F}$ with $a_1\ne 0$, as required.
\end{proof}


Suppose  that Condition (i) holds.  Then the invariance of Condition (iii) w.r.t. transformation \eqref{eq:transformation}  amounts  to the following 

\begin{lemma}
\label{lem:ap2}
 Let $\wt H = \wt H(H, F)$, $d\wt H (P) = 0$ and $\rank (d^2 H)|_{\ker d\mathcal F} = 1$.  Then the conditions  $\rank d^2 H (P) = 3$ and $\rank d^2 \wt H (P) = 3$ are equivalent.
\end{lemma}

\begin{proof}
It is sufficient to use formula \eqref{eq:seconddiff}, namely:
$$
d^2 \wt H = a_1 d^2 H  + 2 a_{22}  dF \otimes dF.
$$ 
In general, these two forms $d^2 \wt H$  and $d^2 H$ do not necessarily have the same rank,  but under the condition that $\rank (d^2 H)|_{\ker d\mathcal F} = 1$   (using $\ker d\mathcal F = \ker dF$), the statements  $\rank d^2 \wt H =3$ and $\rank d^2  H =3$  become equivalent (simple exercise in Matrix Algebra). Indeed,  if we choose a basis  $e_1, e_2, e_3, e_4$ in such a way that $e_1, e_2, e_3$ span $\ker dF$ and $e_2, e_3$ span $\ker d^2 H|_{\ker dF}$ and, in addition, $dF(e_4)=1$ we will see  that in matrix terms, the above formula \eqref{eq:seconddiff} can be rewritten as
$$
d^2 H = \begin{pmatrix} \alpha & 0 & 0 & \beta \\
0 & 0 & 0 & \gamma \\
0 & 0 & 0 & \delta \\
\beta & \gamma & \delta & \lambda
\end{pmatrix},  \qquad d^2 \wt H = a_1  \begin{pmatrix} \alpha & 0 & 0 & \beta \\
0 & 0 & 0 & \gamma \\
0 & 0 & 0 & \delta \\
\beta & \gamma & \delta & \lambda 
\end{pmatrix} + 2a_{22}\begin{pmatrix}
0 & 0 & 0  & 0 \\
0 & 0 & 0 & 0 \\
0 & 0 & 0 & 0 \\
0 & 0 & 0 & 1
\end{pmatrix}.
$$
Now it is easy to see that both statements $\rank d^2 H = 3$ and $\rank d^2 \wt H = 3$ are equivalent to the condition $(\gamma,\delta)\ne (0,0)$, which completes the proof. \end{proof}

Finally, we need to verify the invariance of Condition (ii) w.r.t. transformation \eqref{eq:transformation}.
We first show that this condition does not change if we change $F$.

Consider a new function  $\wt F = \wt F ( F , H)$.   Since scaling $F\mapsto \mathrm{const}\cdot F$ does not affect the surface $\{ F = F(P)\}$, we may assume that  $\frac{\partial \wt F} {\partial F}|_{\mathcal F(P)} = 1$.    

According to the definition of $v^3 H_0$ we need to differentiate the same function $H$ but along two different curves  $\gamma(t) \subset \{ F = F(P)\}$ and $\tilde \gamma (t)\subset \{ \wt F = \wt F(P)\}$.

Let us choose local coordinates $x_1, \dots, x_n$ on $M$ in such a way that $x_1 = F$  and $P=(0,\dots, 0)$. It is easy to see that if we set $\tilde x_1 = \wt F$, then still 
$\tilde x_1, x_2, \dots , x_n$ is a good coordinate system. Moreover, the Jacobi matrix of the corresponding transformation at $P$ is the identity.

In coordinates $x_1, \dots, x_n$, the curve $\gamma(t)$ can be defined as  $\gamma(t) = (0, x_2(t), \dots, x_n(t))$.  The curve $\tilde \gamma(t)$ in the same coordinate system will be defined as 
$$
\tilde \gamma (t) = (x_1(t),  x_2(t), \dots, x_n(t)) = \gamma(t) + (x_1(t), 0, \dots, 0)
$$
(all functions are the same except for  $x_1(t)$ which should be chosen in such a way that  $\wt F = 0$ along the curve.  In other words, $x_1$ can be found as a function of the other variables $x_2,\dots, x_n$ from the implicit relation
$$
0=\wt F ( F, H)=\wt  F \bigl(x_1,  H(x_1, x_2,\dots, x_n)\bigr) \quad  \Leftrightarrow \quad  x_1=g (x_2,\dots, x_n)  
$$  
and correspondingly $x_1(t) = g\bigl(x_2(t), \dots, x_n(t)\bigr)$.

Our goal is to show that 
\begin{equation}
\label{eq:2}
\frac{d^3}{dt^3} |_{t=0} H (\tilde \gamma (t)) = \frac{d^3}{dt^3} |_{t=0} H (\gamma (t)).
\end{equation}

It is easy to see that the Taylor expansion of  $x_1=g (x_2,\dots, x_n)$ starts with quadratic terms.  Indeed, if   (w.l.o.g. we assume that  $F(P)=0$ and $H(P)=0$  so that $\Delta F = F$ and $\Delta H = H$)
$$
\wt F (F, H) =   F + b_2 H + b_{11} F^2 + 2 b_{12} FH + b_{22} H^2 + \dots
$$
then we need to resolve the equation (with respect to $x_1$)
\begin{equation}
\label{eq:1}
0= x_1 + b_2 H + b_{11} x_1^2 + \dots =  x_1 + b_2  \frac{1}{2} \sum_{i,j=1}^n \frac{\partial^2 H}{\partial x_i\partial x_j} x_ix_j + b_{11} x_1^2 +\dots  
\end{equation}

These quadratic terms are sufficient to reconstruct the quadratic terms of the Taylor expansion of the function $x_1=g(x_2,\dots, x_n)$.   This can be done by using the implicit function theorem, but we can also use the substitution
$$
x_1 = g(x_2, \dots, x_n) = \sum_{i,j=2}^n  c_{ij} x_i x_j + \dots
$$
into \eqref{eq:1}.  If we collect  (after substitution) all quadratic terms we obtain:
$$
0 = \sum_{i,j=2}^n  c_{ij} x_i x_j + b_2  \frac{1}{2} \sum_{i,j=2}^n \frac{\partial^2 H}{\partial x_i\partial x_j} x_ix_j  + \dots
$$
which means that up to a constant factor the quadratic expansions of $g$ and $H_0$ coincide:
$$
g(x_2,\dots, x_n) = - \frac{b_2}{2}  \sum_{i,j=2}^n \frac{\partial^2 H}{\partial x_i\partial x_j} x_ix_j + \dots
$$
or
\begin{equation}
\label{eq:3}
d^2 g (\xi, \xi) = - b_2 d^2 H (\xi, \xi)  \quad \mbox{for any $\xi \in \ker dF=\ker d\mathcal F$.}
\end{equation}

Now we are ready to verify \eqref{eq:2}. We have:
$$
\frac{d^3}{dt^3} |_{t=0} H (\tilde \gamma (t))  = 
d^3H (\tilde\gamma', \tilde\gamma',\tilde\gamma') + 3d^2H (\tilde \gamma',\tilde \gamma'') + dH (\tilde \gamma'''),
$$
and 
$$
\frac{d^3}{dt^3} |_{t=0} H (\gamma (t))  = 
d^3H (\gamma', \gamma',\gamma') + 3d^2H(\gamma',\gamma '') + dH(\gamma''').
$$

Since $\gamma' = \tilde\gamma' = v$ and $dH (P) = 0$, we only need to compare the middle terms. In the second relation this term vanishes because  $\gamma'' \in \ker dF$ (as the first component of $\gamma(t)$ identically vanishes) and $\gamma'$ belongs to the kernel of $(d^2H)|_{\ker dF}$.

Consider the difference between $d^2H (\tilde \gamma',\tilde \gamma'')$ and $d^2H ( \gamma', \gamma'')$.  Since $\tilde \gamma' = \gamma'$, we have 
$$
d^2H (\tilde \gamma',\tilde \gamma'') - d^2H ( \gamma', \gamma'') = d^2 H (\gamma', (\tilde \gamma - \gamma)'').
$$
But $\tilde\gamma - \gamma = (g(x_2(t), \dots, x_n(t)), 0, \dots, 0)$ so that for the potentially non-zero component of $\tilde\gamma - \gamma$ we get
$$
\frac{d^2}{dt^2}|_{t=0} g(x_2(t), \dots, x_n(t)) =  d^2 g (\gamma', \gamma') = \mbox{(see \eqref{eq:3})} = - b_2 \, d^2 H (\gamma',\gamma').
$$
It remains to notice that $d^2 H (\gamma',\gamma') = 0$ as $v=\gamma' \in \ker (d^2 H)|_{\ker d\mathcal F}$. Thus, $(\tilde \gamma - \gamma)''=0$.

Thus, we have shown that condition (iii) does not depend on the choice of $F$ (keeping $H$ fixed). The last step is to show that  (iii) does not depend on the choice of $H$  (keeping $F$ fixed).  

Again we use the Taylor expansion \eqref{eq:TaylorH}, but now up to third order terms
$$
\Delta \wt H = a_1 \Delta H +  2 a_{12} \Delta H \Delta F + a_{22} (\Delta F)^2 + a_{222}(\Delta F) + \dots, \quad a_1\ne 0.
$$
We do not need other terms  (like $(\Delta H)^2$ for example)  as in local coordinates $\Delta H$ starts with quadratic terms.  This formula shows that under the condition $\Delta F = 0$,  the Taylor expansions of $H$ and $\wt H$  (up to cubic terms) are proportional with a non-zero factor.  In particular for any curve $\gamma(t)$  lying on the surface $\{ F = F(P)\} = \{ \Delta F = 0\}$ we have 
$$
\frac{d^3}{dt^3}|_{t=0} \wt H(\gamma(t)) = a_1  \frac{d^3}{dt^3}|_{t=0} H(\gamma(t)) \quad \mbox{or equivalently} \quad
v^3 \wt H_0 = a_1  v^3  H_0,
$$
as needed.   This shows that (ii)  is invariant under transformations \eqref{eq:transformation} completing the proof of Proposition \ref{prop:ap1}.
\end{proof}

We also want to explain one important phenomenon mentioned in the proof of Proposition  \ref{prop:5.1}: for hyperbolic points, this coefficient in front of logarithm is known to be proportional to the period of the second (invisible in the real setting) cycle on the complex leaf $\mathcal L_{H,\lambda}$.  More rigorously, this statement can be formulated as follows.

Consider an analytic integrable system with one degree of freedom with the Hamiltonian of the form  $H=xy$ and symplectic structure $\omega = f(x,y) dx\wedge dy$.

Thinking of $x$ and $y$ as real variables, consider one-parameter family of curves  
$$
\gamma_H = \{ xy = H, \ 0<x\le 1, \ 0<y \le 1\}
$$
and the function   (cf. Section \ref{sec:4})
$$
\Pi (H) = \int_{\gamma_H} \frac {\omega}{dH}\, .
$$

It can be easily checked by an explicit computation that this function has the following asymptotics at zero:
$$
\Pi(H) = a(H) \ln H + b(H),
$$
where $a(H)$ and $b(H)$ are both real-analytic in a neighborhood of zero. Similar to Section \ref{sec:4}, here we integrate along a trajectory of the Hamiltonian flow between two sections $N_1 = \{y=1\}$ and $N_2 = \{x=1\}$.    If we change these sections, then $b(H)$ changes too whereas $a(H)$ remains the same so that $a(H)$ has an invariant meaning, i.e., does not depend on the choice of local coordinates $(x,y)$.

On the other hand if we think of $x$ and $y$ as complex variables, then the level $\{ xy = H\}$, from the real viewpoint, is a surface locally homeomorphic to a cylinder  which contains a non-trivial cycle of the form
$$
\hat \gamma_H =\{ x(t) = H e^{it}, \  y(t) = e^{-it} \},
$$  
so that we can introduce another function 
$$
\hat \Pi (H) = \int_{\hat \gamma_H} \frac {\omega}{dH}.
$$
This function is analytic in $H$.

The relation between $\Pi (H)$ and $\hat \Pi (H)$ is given by the following

\begin{proposition}
\label{prop:ap2}
We have $a(H) = \pm \frac{1}{2\pi i} \hat \Pi (H) $  or equivalently,
$$
\Pi (H) =  \pm \frac{1}{2\pi i} \hat \Pi (H) \ln H  + b(H)
$$
with $b(H)$ being analytic\footnote{The sign $\pm$ reflects the fact that both $\Pi (H)$ and $\hat \Pi (H)$ depend on the choice of orientations on the curves $\gamma_H$ and $\hat\gamma_H$.}.
\end{proposition}

\begin{proof}
One can proof this fact by using monodromy arguments  (which is nice and conceptual), but we will use a well known ``isochore Morse lemma'' \cite{CVey}  which allows to get this result by an explicit computation:  there exists a local coordinate system such that
$$
H = xy \qquad \mbox{and} \qquad \omega = f(H) dx\wedge dy.
$$

The form $\dfrac{\omega}{dH}$  can be replaced by the form $ f(H) y^{-1} dy$ and we obtain:
$$
\hat \Pi (H) = \int_{\hat \gamma_H} \frac {\omega}{dH} = \int_{\hat \gamma_H} f(H) y^{-1} dy = f(H) \int_0^{2\pi}  ( e^{-it} )^{-1} d ( e^{-it} ) =
-2\pi  i  f(H)
$$

On the other hand,   $\gamma_H$ can be parametrised as  $y=t$,  $x=H t^{-1}$, $t\in [H, 1]$ and we get:
$$
\Pi (H) = \int_{\gamma_H} \frac {\omega}{dH} = \int_{\gamma_H}  f(H) y^{-1} dy = f(H) \int_H^1 \frac{dt}{t} = - f(H) \ln H. 
$$

Comparing these formulas for $\hat \Pi (H) $ and $\Pi (H)$  gives the required result.  \end{proof}

\bibliographystyle{plain}
\nocite{*} 
\bibliography{references}

\end{document}